\documentclass[11pt,reqno]{amsart}
\usepackage[a4paper]{anysize}\marginsize{2.5cm}{2.5cm}{2cm}{2cm}
\pdfpagewidth=\paperwidth \pdfpageheight=\paperheight
\allowdisplaybreaks
\usepackage{mlmodern}

\usepackage[english]{babel}
\usepackage[latin1]{inputenc}
\usepackage[T1]{fontenc}
\usepackage{amssymb,amsmath,amstext,amsfonts,amsthm,amscd,latexsym, mathrsfs}
\usepackage{mathtools}
\usepackage{verbatim}
\usepackage{relsize}
\usepackage[colorlinks=true, urlcolor=blue, linkcolor=blue, citecolor=blue]{hyperref}
\usepackage{tikz-cd}
\usepackage{caption}
\usepackage{graphicx}
\usepackage{xcolor}

\numberwithin{equation}{section}
\theoremstyle{plain}
\newtheorem*{theorem*}{Theorem}
\newtheorem{theorem}{Theorem}
\numberwithin{theorem}{section}
\newtheorem{proposition}[theorem]{Proposition}
\newtheorem{lemma}[theorem]{Lemma}
\newtheorem{corollary}[theorem]{Corollary}

\newtheorem{remark}[theorem]{Remark}

\theoremstyle{definition}
\newtheorem{definition}[theorem]{Definition}
\newtheorem{example}[theorem]{Example}
\newtheorem{notation}[theorem]{Notation}

\newcommand{\C}{{\mathbb C}}
\newcommand{\Q}{{\mathbb Q}}

\newcommand{\G}{{\mathbb G}}
\newcommand{\K}{{\mathbb K}}
\newcommand{\R}{{\mathbb R}}
\newcommand{\Z}{{\mathbb Z}}
\newcommand{\N}{{\mathbb N}}
\newcommand{\V}{{\mathbb V}}
\newcommand{\PP}{{\mathbb P}}

\def\codim{\operatorname{codim}}
\def\image{\operatorname{im}}
\def\rank{\operatorname{rank}}
\newcommand{\mT}{\mathsmaller{\mathsf{T}}}

\newcommand{\mle}{\mathsmaller{\le}}
\newcommand{\mcp}{\mathsmaller{|P|}}
\newcommand{\sat}{\mathrm{sat}}

\newcommand{\kd}{{\mathcal D}}
\newcommand{\km}{{\mathcal M}}
\newcommand{\kk}{{\mathcal K}}
\newcommand{\kp}{{\mathcal P}}
\newcommand{\kv}{{\mathcal V}}
\newcommand{\ku}{{\mathcal U}}

\makeatletter
\renewcommand*\env@matrix[1][*\c@MaxMatrixCols c]{%
    \hskip -\arraycolsep
    \let\@ifnextchar\new@ifnextchar
    \array{#1}}
\makeatother

\makeatletter
\@namedef{subjclassname@2020}{%
	\textup{2020} Mathematics Subject Classification}
\makeatother

\author{Flavio Salizzoni}
\author{Luca Sodomaco}
\address{Max Planck Institute for Mathematics in the Sciences, Leipzig, Germany}
\email{flavio.salizzoni@mis.mpg.de}
\email{luca.sodomaco@mis.mpg.de}
\author{Julian Weigert}
\address{Max Planck Institute for Mathematics in the Sciences, Leipzig, Germany}
\address{Mathematisches Institut, Universit\"at Leipzig, Augustusplatz 10, 04109 Leipzig, Germany}
\email{julian.weigert@mis.mpg.de}

\title[Nonlinear Kalman varieties]{Nonlinear Kalman varieties}

\subjclass[2020]{Primary: 15A18; Secondary: 13P25, 14N05, 14N10, 14Q20, 93B25}
\keywords{Eigenvector, Kalman's observability condition, determinantal variety, Kalman variety, vector bundle, Chow ring, singular locus}

\begin{document}

\begin{abstract}
We study the locus of square matrices having at least one eigenvector on a prescribed algebraic variety $X$. When $X$ is a linear subspace, this data locus is known as the {\em Kalman variety} of $X$ and was studied first by Ottaviani and Sturmfels. Motivated by recent applications to quantum chemistry and optimization, in this work, we focus on {\em nonlinear Kalman varieties}, that is, Kalman varieties relative to arbitrary projective varieties $X$. We study the basic invariants of these varieties, such as their dimensions, degrees, and singularities.
Furthermore, Ottaviani and Sturmfels provide determinantal equations in the linear case.
We generalize their result to Kalman varieties of hypersurfaces by providing a determinantal-like description of their equation.
\end{abstract}

\maketitle

\section{Introduction}

Given a square matrix $A\in\K^{n\times n}$ with entries in a field $\K$, one of the most classical problems in linear algebra is the study of its eigenvectors and eigenvalues, that is, the solutions $v\in\K^n\setminus\{0\}$ and $\lambda\in\K$ of the system
\begin{equation}\label{eq: linear algebra}
A\,v = \lambda\,v\,.
\end{equation}
In this work, we study loci of matrices $A$ that admit eigenvectors with additional structure. More precisely, we consider a projective algebraic variety $X\subseteq\PP^{n-1}=\PP(\K^n)$ and we look for all square matrices $A$ such that at least one of its {\em eigenpoints}, which are classes $[v]$ of eigenvectors, sits on $X$. We call these loci {\em nonlinear Kalman varieties} and we denote them by $\kk(X)$, see Definition~\ref{def: Kalman}.

Ottaviani and Sturmfels introduced Kalman varieties in~\cite{ottaviani2013matrices} when $X$ is a linear subspace and studied their irreducibility, degree, and their singular loci. Further aspects of equations of Kalman varieties have been addressed in~\cite{sam2012equations,huang2017equations}. More recently, the notion of Kalman variety has been extended to higher-order tensors in~\cite{ottaviani2022tensors,shahidi2023degrees} by considering singular vector tuples rather than eigenpoints. Restricting to symmetric matrices, the nonlinear Kalman varieties considered here coincide with the Kalman varieties studied in \cite{shahidi2023degrees}, see Remark~\ref{rmk: symmetric Kalman}.

Our work is also motivated by two applications coming from quantum chemistry and quadratic optimization, which we summarize here. First, it is worth mentioning that, despite being a reasonably easy problem to study, computing the solutions of \eqref{eq: linear algebra} may be an intractable problem if the size $n$ of $A$ is huge. This happens, for example, in quantum chemistry. In this case, $A$ is a symmetric matrix of size $n=\binom{k}{\ell}$, called the electronic structure Hamiltonian. It encodes the interaction among $\ell$ electrons, obtained via discretization into $k$ spin-orbitals of the electronic Schr\"odinger equation. In this setting, the vector $v$ represents a quantum state. One may easily observe that, when $k$ increases, the problem size of \eqref{eq: linear algebra} grows exponentially. The main goal of {\em coupled cluster theory} is to accurately approximate the solutions of \eqref{eq: linear algebra} with solutions of a much smaller, nonlinear system, whose equations are called {\em coupled cluster (CC) equations}. Further details on the topic and on the algebraic geometry related to CC theory can be found in the seminal paper \cite{faulstich2025algebraic} (with parameters $(n,d)$ replacing our $(k,\ell)$). In brief, the CC equations are formed by a subset (encoded by a parameter $\sigma\subseteq[\ell]=\{1,\ldots,\ell\}$) of the $2\times 2$ minors of the $n\times 2$ matrix with columns $A\,v$ and $v$, together with the condition that $[v]\in V_\sigma$, where $V_\sigma$ is a so-called {\em truncation variety} (see \cite[Eq. (26)]{faulstich2025algebraic}):
\begin{equation}\label{eq: CC equations}
(A\,v)_\sigma = \lambda\,v_\sigma\,,\quad [v]\in V_\sigma\,.  
\end{equation}
The choice of minors and of $V_\sigma$ ensures a polynomial system \eqref{eq: CC equations} with finitely many solutions $[v]\in\PP^{n-1}$ for a sufficiently generic choice of $A$. The size of the new problem \eqref{eq: CC equations} is considerably smaller than the size of \eqref{eq: linear algebra}, at the price of computing only approximate eigenvectors of $A$ \eqref{eq: linear algebra}, and of obtaining also nonreal solutions in complex-conjugated pairs, a phenomenon that cannot occur in \eqref{eq: linear algebra} if $A$ is chosen real and symmetric, thanks to the Spectral Theorem.

Zooming out from quantum chemistry, one observes that the nonzero solutions $v$ of \eqref{eq: linear algebra} correspond to the critical points of the Rayleigh quotient $\frac{v^\mT Av}{v^\mT v}$. Equivalently, the normalized eigenvectors of $A$ are the critical points of the quadratic homogeneous polynomial $f_A(v)=v^\mT Av$ constrained to the sphere. In several applications, it is natural to study the critical points of $f_A$ constrained to the sphere and to another subset $X$ cut out by polynomial equations. The minimization of the Rayleigh quotient constrained to a linear subspace $X$ is sometimes referred to as the {\em constrained eigenvalue problem}. In general, the variety of constraints may not be linear. For example, the minimization of the Rayleigh quotient with quadratic constraints is studied in \cite{karow2011values,prajapati2022optimizing}. Going back to quantum chemistry, an interesting case is when $X$ is a Grassmannian, a Segre variety of rank-one tensors, or more generally, a tensor network variety. In the recent work \cite{salizzoni2025nonlinear}, we introduced the {\em Rayleigh-Ritz (RR) degree of index $\omega$} of an algebraic variety $X$, as the number of complex critical points over $X$ of a polynomial function of degree $\omega$ with generic coefficients. The equations defining the critical points of a constrained Rayleigh quotient optimization may be called {\em RR equations}, mimicking the CC equations of quantum chemistry. The case $\omega=2$ corresponds precisely to the normalized Rayleigh quotient optimization constrained to $X$. A further detailed study on RR degrees of index $\omega=2$ of tensor train varieties is developed in \cite{borovik2025numerical}.

The RR and CC equations answer two different questions; therefore, they yield two very different invariants, the {\em RR degree} and the {\em CC degree}: On the one hand, the solutions of \eqref{eq: CC equations} are only approximations of \eqref{eq: linear algebra}. On the other hand, the solutions of RR equations with respect to $X$ are called $X$-eigenvectors, and for a generic data matrix $A$, they coincide with the classical eigenvectors only when $X=\PP^{n-1}$. Returning to the main goal of studying nonlinear Kalman varieties, our research is also driven by the following question: {\em For which data matrices $A$ does there exist a solution $v$ of either the CC equations \eqref{eq: CC equations} or the RR equations which is also a solution of \eqref{eq: linear algebra}, that is, $v$ is eigenvector of $A$?}

We now summarize the results of this work. In Section~\ref{sec: definition}, we set up the main notations and definitions used in the paper, and we state the first property of nonlinear Kalman varieties, such as their irreducibility if the fixed variety $X$ is irreducible, their codimension in the space of matrices, and their degree. A first tentative description of the equations of Kalman varieties is done in Example \ref{ex: Kalman variety plane conic} for a plane conic $X$.

In Section~\ref{sec: Kalman matrix}, we recall the construction of the Kalman matrix associated with a linear subspace $L=\PP(\ker(C))$ for some $C\in \R^{(n-m)\times n}$. The definition originates from control theory. In this context, given $A\in \R^{n\times n}$ and $C$ as above, a discrete-time system in state-space form 
\[
    \begin{cases}
        x(k+1) = A\,x(k)\\
        y(k) = Cx(k)
    \end{cases}
\]
is {\em observable} if the initial state $x(0)$ can be determined by the output vectors $y(0),\ldots,y(n-1)$. Such a system is observable if and only if the associated {\em Kalman matrix} $K(C)$ in \eqref{eq: Kalman matrix} has full rank. This result is known as {\em Kalman's Observability Condition}, see~\cite{kalman1960contributions}.
For a fixed $C$, the set of all matrices $A$ for which the corresponding discrete-time system is unobservable is precisely the Kalman variety $\kk(L)$, and it is cut out by all maximal minors of the Kalman matrix $K(C)$, see also \cite[Prop. 1.1]{ottaviani2013matrices}. In particular, Kalman varieties of linear subspaces are {\em determinantal}, i.e., their equations are minors of the Kalman matrix, and similarly for their singular loci, as described in \cite[Thm. 4.5]{ottaviani2013matrices}.
Our goal is to describe the equations that define a nonlinear Kalman variety. In particular, inspired by the linear case, we are interested in understanding whether these equations can be interpreted as maximal minors of a certain matrix.
In Definition~\ref{def: nonlinear Kalman matrix}, for any positive integer $d$, we introduce the {\em Kalman matrix of order $d$} and we denote it by $K_d(C)$. This matrix is essentially obtained by replacing the original matrix $A$ in the classical Kalman matrix $K(C)$ with its $d$th symmetric power, as defined in Definition~\ref{def: symmetric power matrix}. What is more, the matrix $C=C_X$ is formed using the coefficients of the polynomials in a minimal generating set for the ideal of the variety $X$. The locus where all maximal minors of the new Kalman matrix $K_d(C_X)$ vanish contains the nonlinear Kalman variety $\kk(X)$, but in general the containment is strict.

In Section~\ref{sec: Kalman hypersurfaces}, we restrict to the case when $X$ is a hypersurface cut out by a single homogeneous polynomial of degree $d$. In this case, the nonlinear Kalman matrix is square, and its determinant has several factors with multiplicities, including the equation of $\kk(X)$. After several preliminary results, in Theorem~\ref{thm: factors determinant Kalman matrix}, we explicitly describe all factors of the determinant and their multiplicities.

Finally, in Section~\ref{sec: singular loci}, we study the singularities of nonlinear Kalman varieties. Singular loci of nonlinear Kalman varieties are particularly hard to study, as their geometry strongly depends on the singularities of the fixed variety $X$. We provide a first geometric intuition of this fact in Theorem~\ref{thm: decomp sing locus Kalman}. This first result is later used to determine the codimension and the degree of the reduced singular locus of the Kalman variety of a nonsingular hypersurface $X$ of degree $d$, leading to Theorem~\ref{thm: degree singular locus Kalman smooth hypersurface}. The main idea behind our result is to degenerate $X$ to a union of hyperplanes $X'=H_1\cup\cdots\cup H_d$, study the reduced singular locus of $\kk(X')$ instead, and finally track back the contribution in the degree of the latter singular locus that does not depend on the singularities of $X'$.  

Supplementary \verb|Macaulay2| \cite{grayson1997macaulay2} software used in the computational examples of this paper can be found at the Zenodo repository \cite{salizzoni2025supplementary}.

\section{Definition and first properties}\label{sec: definition}

We always work over the field of complex numbers $\C$, but we stress that most of our results can be extended to an arbitrary algebraically closed field $\K$ of characteristic zero. More precisely, the property of being of characteristic zero is used in all degree computations of Section \ref{sec: Kalman hypersurfaces}. Furthermore, working over $\C $ is essential in the proof of Lemma~\ref{lem: discriminant_factors_with_multi}.
We start by fixing the notations used throughout the paper.

\begin{notation}\label{notations}
For an integer $n\ge 1$, we use the shorthand $[n]$ for the set $\{1,\ldots,n\}$.
If not otherwise specified, we consider a vector $v=(v_1,\ldots,v_n)^\mT\in\C^n$ as a column vector. We usually denote by $x=[v]=[v_1:\cdots :v_n]$ a point in $\PP^{n-1}=\PP(\C^n)$ for some $v\in\C^n\setminus\{0\}$.
We denote by $\C^{n\times n}$ the space of $n\times n$ matrices $A=(a_{ij})$ with entries $a_{ij}\in\C$, and with $\PP^{n^2-1}=\PP(\C^{n\times n})$ its projectivization.
We adopt the shorthand $\C[x]\coloneqq\C[x_1,\ldots,x_n]$ for the ring of polynomials in $x_1,\ldots,x_n$ with complex coefficients, and similarly $\C[A]\coloneqq\C[a_{ij}\mid 1\le i,j\le n]$. Furthermore, we write $\C[x]_d\coloneqq\C[x_1,\ldots,x_n]_d$ for any integer $d\ge 1$ to denote the vector space of homogeneous polynomials of degree $d$ in $x_1,\ldots,x_n$. In particular, its dimension is $N\coloneqq\binom{n-1+d}{d}$.
For any integer $d\ge 1$,  we define the Veronese embedding $\nu_d\colon\PP^{n-1}\hookrightarrow\PP(\C[x]_d)$ sending the class $[v]=[v_1:\cdots:v_n]\in\PP^{n-1}$ to the class $[(v^*)^d]\in\PP(\C[x]_d)$, where $v^*=v_1x_1+\cdots+v_nx_n\in\C[x]_1$ is the linear form with coefficient vector $v$. We also denote by $\nu_d(v)\in\C[x]_d\cong\C^N$ the column vector of coordinates of $(v^*)^d$ with respect to a fixed monomial order $\prec$.
Finally, we denote by $X$ a (reduced but not necessarily irreducible) projective variety in $\PP^{n-1}$, and we set $m$ to be the dimension of the affine cone in $\C^n$ over $X$, or equivalently $m-1=\dim X$. We always assume that $X\subsetneq\PP^{n-1}$.
\end{notation}

Throughout the paper, we say that $x\in\PP^{n-1}$ is an {\em eigenpoint of $A\in\C^{n\times n}$} if $x=[v]$ for some eigenvector $v$ of $A$.

\begin{definition}\label{def: Kalman}
Let $X\subseteq\PP^{n-1}$ be an algebraic variety of projective dimension $m-1$. We define the {\em (nonlinear) Kalman variety of $X$} as
\[
    \kk(X) \coloneqq \{[A]\in\PP^{n^2-1}\mid\text{$\exists\,x\in X$ such that $x$ is an eigenpoint of $A$}\}\,.
\]
\end{definition}

The dimension and degree of the Kalman variety of an irreducible variety $X$ can be easily expressed in terms of the dimension and degree of $X$. The following result is an almost immediate generalization of \cite[Prop. 1.2]{ottaviani2013matrices}.

\begin{proposition}\label{prop: dim deg kalman}
Let $X\subseteq\PP^{n-1}$ be an irreducible algebraic variety of dimension $m-1$. 
The Kalman variety $\kk(X)$ is irreducible in $\PP^{n^2-1}$ of codimension $n-m$ and degree $\deg X\cdot\binom{n}{m-1}$.
\end{proposition}
\begin{proof}
Consider the incidence variety
\begin{equation}\label{eq: incidence variety Sigma}
    \Sigma(X) \coloneqq \{([A],x)\in\PP^{n^2-1}\times\PP^{n-1}\mid\text{$x\in X$ is an eigenpoint of $A$}\}\,.
\end{equation}
Let $\pi_2\colon\Sigma(X)\to\PP^{n-1}$ be the morphism induced by the projection onto the second factor. Then $\pi_2(\Sigma(X))=X$ and for every $x\in X$, the subset of points $[A]\in\PP^{n^2-1}$ such that $x$ is an eigenpoint of $A$ is linear of dimension $n^2-1-(n-1)=n^2-n$. By the theorem on the dimension of fibres, we conclude that $\Sigma(X)$ is irreducible of dimension $n^2-1-(n-m)$. Now consider the morphism $\pi_1\colon\Sigma(X)\to\PP^{n^2-1}$ induced by the projection onto the first factor. Then $\pi_1(\Sigma(X))=\kk(X)$, in particular $\kk(X)$ is irreducible. Furthermore, a generic $[A]\in\kk(X)$ admits precisely one eigenpoint $x$ lies on $X$:
Indeed, since $X\subsetneq\PP^{n-1}$, there exists a hyperplane $H\subseteq\PP^{n-1}$ such that $X\nsubseteq H$. One can choose $n$ points $x_1=[v_1]\in X\setminus H$ and $x_i=[v_i]\in H\setminus X$ for all $i\in\{2,\ldots,n\}$ such that $H=\PP(\langle v_2,\ldots,v_n\rangle)$. If $V$ denotes the $n\times n$ matrix whose columns are $v_1,\ldots,v_n$ and $D$ is a diagonal matrix with generic diagonal entries, then $x_1$ is the unique eigenpoint of $A=VDV^{-1}\in\C^{n\times n}$ on $X$.
As the condition of having two or more eigenvectors on $X$ is closed, this is the generic behavior.
This implies that the restriction of $\pi_1$ to the image $\kk(X)$ is a birational morphism, in particular it is finite-to-one, hence $\dim\kk(X)=\dim\Sigma(X)=n^2-1-(n-m)$, or equivalently $\codim\kk(X)=n-m$. The computation of $\deg\kk(X)$ is done similarly to that in the proof of \cite[Prop. 1.2]{ottaviani2013matrices}.
\end{proof}

\begin{remark}\label{rmk: many components}
In the previous result, $X$ may be replaced by a reduced but not necessarily irreducible variety for the cost of giving up irreducibility of $K(X)$: if $X=X_1\cup\cdots\cup X_k$ is an irreducible decomposition of $X$, then $\kk(X)=\kk(X_1)\cup\cdots\cup\kk(X_k)$. The statements about dimension and degree remain true.\end{remark}

To compute the ideal of $\kk(X)$, first consider the $n\times 2$ matrix
\begin{equation}\label{eq: matrix M}
M = 
\begin{pmatrix}[c|c]
A\,x & x    
\end{pmatrix}
=
\begin{pmatrix}[c|c]
a_{11}x_1+\cdots+a_{1n}x_n & x_1 \\
\vdots & \vdots \\
a_{n1}x_1+\cdots+a_{nn}x_n & x_n 
\end{pmatrix}\,.
\end{equation}
Let $I(X)$ denote the radical ideal of $X$. The ideal of the incidence variety $\Sigma(X)$ in \eqref{eq: incidence variety Sigma} is
\begin{equation}\label{eq: ideal I}
I \coloneqq \left[I(X)+\left(\text{$2\times 2$ minors of $M$}\right)\right]\colon (x_1,\ldots,x_n)^\infty \subseteq\C[x]\otimes\C[A]\,,
\end{equation}
and the ideal of $\kk(X)$ is the intersection $I\cap\C[A]$, corresponding to the elimination of the variables $x_1,\ldots,x_n$ from $I$. 

\begin{remark}\label{rmk: saturation}
We show that the saturation in \eqref{eq: ideal I} with respect to the ideal $(x_1,\ldots,x_n)$ is necessary, while the saturation with respect to $(a_{ij}\mid 1\le i,j\le n)$ is not. 
Let $I'\coloneqq I(X)+(\text{$2\times 2$ minors of $M$})$ and define $\Sigma_{\mathrm{aff}}'(X)$ and $\Sigma_{\mathrm{aff}}(X)$ as the zero loci in $\C^n\times\C^{n\times n}$ of $I'$ and $I$, respectively.
On the one hand $\{0\}\times\C^{n\times n}$ is a component of $\Sigma_{\mathrm{aff}}(X)$, unless $X=\PP^{n-1}$, and $\Sigma_{\mathrm{aff}}(X)=\overline{\Sigma_{\mathrm{aff}}'(X)\setminus(\{0\}\times\C^{n\times n})}\subsetneq\Sigma_{\mathrm{aff}}'(X)$, where the closure is taken in the Zariski topology. On the other hand, for every pair $(v,0)\in \Sigma_{\mathrm{aff}}(X)$ we have $(v,0)=\lim_{n\to\infty}(v,\frac{1}{n}A)$ for some $A\neq 0$ having $v$ as eigenvector. This shows that $\Sigma_{\mathrm{aff}}(X)=\overline{\Sigma_{\mathrm{aff}}(X)\setminus(\C^n\times\{0\})}$ in the Euclidean topology (and also in the Zariski topology because $\Sigma_{\mathrm{aff}}(X)\setminus(\C^n\times\{0\})$ is a nonempty open dense subset in both topologies).
\end{remark}

Alternative methods to compute the equations of $\kk(X)$ depend on the specific variety $X$ considered. For example, one may apply the theory of resultants to compute the equation of the Kalman variety of a quadric hypersurface, as we do in Example \ref{ex: Kalman variety plane conic}. In particular, we apply a method developed by Salmon, which we summarize in the following theorem.

\begin{theorem}{\cite[Lsn. X]{salmon1876lessons}}\label{thm: Salmon}
Let $f_1,\ldots,f_n$ be $n$ homogeneous polynomials in $n$ variables. Then any solution $x\in\PP^{n-1}$ of the system $\{f_1(x)=\cdots=f_n(x)=0\}$ is also a solution of the Jacobian polynomial
\begin{equation}\label{eq: jacobian polynomial}
J(x) \coloneqq \det
\begin{pmatrix}
    \frac{\partial f_1}{\partial x_1}(x) & \cdots & \frac{\partial f_1}{\partial x_n}(x) \\
    \vdots & & \vdots \\
    \frac{\partial f_n}{\partial x_1}(x) & \cdots & \frac{\partial f_n}{\partial x_n}(x) 
\end{pmatrix}\,.
\end{equation}
Moreover, if the polynomials $f_i$ all have the same degree, then any solution $x\in\PP^{n-1}$ of the system $\{f_1(x)=\cdots=f_n(x)=0\}$ is also a solution of all the polynomials $\frac{\partial J}{\partial x_i}(x)$, $i\in[n]$.
\end{theorem}

\begin{example}\label{ex: Kalman variety plane conic}
Consider a $3\times 3$ matrix $A=(a_{ij})$ and the projective conic $Q\subseteq\PP^2$ of equation
\begin{equation}\label{eq: plane conic with parameters}
f(x) = b_{200}x_1^2+b_{110}x_1x_2+b_{101}x_1x_3+b_{020}x_2^2+b_{011}x_2x_3+b_{002}x_3^2 = 0\,,
\end{equation}
where the $b_{ijk}$'s are additional parameters. Applying Proposition~\ref{prop: dim deg kalman}, the Kalman variety $\kk(Q)$ is a hypersurface of degree $6$ in $\PP(\C^{3\times 3})\cong\PP^8$, namely it is defined by a homogeneous polynomial of degree $6$ in the variables $a_{ij}$. 

Alternatively, for this example, we compute this polynomial via resultants.
Consider the matrix $M$ in \eqref{eq: matrix M} for $n=3$, and let $f_2,f_3$ be the minors of $M$ obtained by selecting the rows $(1,2)$ and $(1,3)$ of $M$, respectively, while $f_3\coloneqq f$. The system $\{f_1(x)=f_2(x)=f_3(x)=0\}$ admits a solution in $\PP^2$ if and only if the {\em resultant} $\mathrm{Res}(f_1,f_2,f_3)$ vanishes \cite[\S13.1]{gelfand1994discriminants}. Note that $f_1,f_2,f_3$ are three quadratic polynomials in $x$, so we can apply both parts of Theorem~\ref{thm: Salmon}. In this case, the Jacobian polynomial $J(x)$ in \eqref{eq: jacobian polynomial} is homogeneous of degree three in $x_1,x_2,x_3$. Let $B$ be the $6\times 6$ coefficient matrix of the vector of polynomials $(\frac{\partial J}{\partial x_1},\frac{\partial J}{\partial x_2},\frac{\partial J}{\partial x_3},f_1,f_2,f_3)$ with respect to the basis of monomials of degree $2$ in $x_1,x_2,x_3$. The determinant of $B$ is equal to $\mathrm{Res}(f_1,f_2,f_3)$, up to a scalar factor. We verified that
\[
\mathrm{Res}(f_1,f_2,f_3) = g_1\,g_2\,,
\]
where $g_1 = b_{002}a_{12}^2-b_{011}a_{12}a_{13}+b_{020}a_{13}^2$ and $g_2$ is a homogeneous polynomial in the entries of $A$ and in the coefficients of $f$ of bi-degree $(6,3)$, with 2832 terms. The polynomial $g_2$ is the equation of the nonlinear Kalman variety $\kk(Q)$. Specializing to the conic of equation $f=x_2^2-x_1x_3=0$, we get $g_1=a_{13}^2$ and $g_2$ is a polynomial with $138$ terms. For the code used, see the script \verb|Kalman_conic_Salmon.m2| available at \cite{salizzoni2025supplementary}.\hfill$\diamondsuit$
\end{example}

\begin{remark}\label{rmk: symmetric Kalman}
One might restrict Definition~\ref{def: Kalman} only to symmetric matrices, and call {\em symmetric Kalman variety} the locus $\kk_{\mathrm{sym}}(X)$ of classes of symmetric matrices $A$ having at least one eigenpoint on $X$.
Symmetric Kalman varieties are instances of {\em (projective) conditional data loci}, which are discussed in \cite{dirocco2024Relative}, see also \cite{horobet2022data}. In general, one considers an affine variety $Y\subseteq\R^n$ equipped with a positive-definite quadratic form, and a subvariety $Z\subseteq Y$. The goal is to compute the (complex) critical points of the distance function from a given data point $u\subseteq\R^n$, restricted to $Y$. A generic $u$ yields finitely many complex critical points, whose number is the {\em distance degree of $Y$} with respect to the quadratic form chosen (see \cite{draisma2016euclidean}). In this generality, none of the critical points sits on the subvariety $Z$. The conditional data locus of $Y$ given $Z$ is the subvariety $\mathrm{DL}_{Y|Z}$ of data points such that at least one of such critical points belongs to $Z$, see \cite[Dfn. 5.6]{dirocco2024Relative}. Translated into our setting, the ambient space is the space of real symmetric matrices $\mathrm{Sym}^2\R^n\subseteq\R^{n\times n}$ equipped with the Bombieri-Weyl inner product, and $Y$ is the affine cone over the Veronese variety $\nu_2(\PP^{n-1})$, that is the cone of symmetric matrices of rank one. It is classically known (as recalled in \cite[Prop. 2.6]{salizzoni2025nonlinear} for symmetric tensors of arbitrary degree) that the normalized real eigenvectors of a real symmetric matrix $A\in\mathrm{Sym}^2\R^n\subseteq\R^{n\times n}$ are in one-to-one correspondence with the rank-one symmetric matrices which are critical for the Bombieri-Weyl distance function from $A$ restricted to the affine cone of rank-one symmetric matrices. Finally, one considers $Z$ to be the affine cone over $\nu_2(X)$ for some variety $X\subseteq\PP^{n-1}$. With the notation used in \cite{dirocco2024Relative}, we then have
\[
\kk_{\mathrm{sym}}(X) = \mathrm{DL}_{\nu_2(\PP^{n-1})|\nu_2(X)}\,,
\]
where the right-hand side is intended as a projective variety.
Therefore, the membership problem $[A]\in\kk_{\mathrm{sym}}(X)$ can be framed in the context of polynomial optimization, in particular in metric algebraic geometry \cite{breiding2024metric}. A similar metric description of Kalman varieties can also be done for partially symmetric tensors and singular vector tuples. For more details on this subject, see \cite[Rmk. 7.2]{dirocco2024Relative} and \cite{shahidi2023degrees}. In particular, the degree of $\kk_{\mathrm{sym}}(X)$ is a special case of \cite[Thm. 1]{shahidi2023degrees} (with $k=1$ and $\omega=2$ in the reference), see also \cite[Rmk. 23]{shahidi2023degrees}.

For example, one may compute the equation of $\kk_{\mathrm{sym}}(Q)\subseteq\PP(\mathrm{Sym}^2\C^3)$ for the conic $Q$ of Example \ref{ex: Kalman variety plane conic}. In that case $\kk_{\mathrm{sym}}(Q)$ is also a hypersurface of degree six cut out by a polynomial with 99 terms in the variables of $A\in\C^{3\times 3}$.
\end{remark}

\section{Kalman matrix for nonlinear Kalman varieties}\label{sec: Kalman matrix}

In this section, we introduce the notion of Kalman matrix for nonlinear Kalman varieties. The underlying idea is to trace the problem back to the linear case using the Veronese embedding $\nu_d$ introduced in Notation \ref{notations}. We start by recalling the construction of the Kalman matrix in the linear case.

Let $L\subseteq\PP^{n-1}$ be a linear subspace of projective dimension $m-1$. It can be described as $L=\PP(\ker(C))$ for some full rank matrix $C\in\C^{(n-m)\times n}$. As stated in \cite[Prop. 1.1]{ottaviani2013matrices}, the Kalman variety $\kk(L)$ is {\em determinantal}, in particular, it coincides with the variety cut out by the $n\times n$ minors of the {\em Kalman matrix}
\begin{equation}\label{eq: Kalman matrix}
    K(C) \coloneqq 
    \begin{pmatrix}
        C\\
        CA\\
        CA^2\\
        \vdots\\
        CA^m
    \end{pmatrix}\in\C[A]^{(n-m)(m+1)\times n}\,.
\end{equation}
One might wonder if a similar determinantal description holds for nonlinear Kalman varieties. The intuition comes from the case of a hypersurface $X=\V(f)\subseteq\PP^{n-1}$ for a given $f\in\C[x]_d$. Indeed, the condition $f=0$ corresponds to a hyperplane section of the Veronese variety $\nu_d(\PP^{n-1})\subseteq\PP(\C[x]_d)$. This suggests considering the Kalman matrix of the hyperplane $C_f\subseteq\PP(\C[x]_d)$ associated to $f$, but replacing $A$ with a larger matrix which we introduce below.

\begin{definition}\label{def: symmetric power matrix}
Let $n$ and $d$ be positive integers. Denote by $\rho_d\colon\mathrm{GL}(\C^n)\to\mathrm{GL}(\C[x]_d)$ the $d$th symmetric power representation of $\mathrm{GL}(\C^n)$, defined by $\rho_d(g)(f(x))=f(g^{-1}(x))$ for all $f\in\C[x]_d\cong\C^N$ and $g\in \mathrm{GL}(\C^n)$. This is a polynomial representation, and so the map $\rho_d$ can be extended to all the matrices in $\C^{n\times n}$ (see \cite[Ch. 4]{sturmfels1993algorithms}). For any $A\in\C^{n\times n}$, we define the {\em $d$th symmetric power of $A$} as the matrix $\rho_d(A)\in\C^{N\times N}$.
\end{definition}

\begin{example}\label{ex: symmetric power matrix}
Consider the matrix $A=(a_{ij})\in\C^{3\times 3}$. To compute its symmetric power $\rho_2(A)\in\C^{6\times 6}$, we apply the representation $\rho_2$ of Definition~\ref{def: symmetric power matrix} to the basis of monomials $\{x_{1}^{2},\:x_{1}x_{2},\:x_{1}x_{3},\:x_{2}^{2},\:x_{2}x_{3},\:x_{3}^{2}\}$ of $\C[x]_2$. For example, if we choose the first basis element $x_1^2$, we need to write the vector of coefficients in the form
\[
\resizebox{\textwidth}{!}{
$\begin{aligned}
(A\,x)_1^2 = (a_{11}x_1+a_{12}x_2+a_{13}x_3)^2 = a_{11}^{2}x_{1}^{2}+2\,a_{11}a_{12}x_{1}x_{2}+2\,a_{11}a_{13}x_{1}x_{3}+a_{12}^{2}x_{2}^{2}+2\,a_{12}a_{13}x_{2}x_{3}+a_{13}^{2}x_{3}^{2}\,.
\end{aligned}$}
\]
Repeating this procedure for all basis elements of $\C[x]_2$, we obtain that
\begin{equation}\label{eq: symmetric power 3x3 matrix}
\resizebox{0.92\textwidth}{!}{
$\begin{aligned}
\rho_2(A) =
\begin{pmatrix}
      a_{11}^{2}&2\,a_{11}a_{12}&2\,a_{11}a_{13}&a_{12}^{2}&2\,a_{12}a_{13}&a_{13}^{2}\\
      a_{11}a_{21}&a_{12}a_{21}+a_{11}a_{22}&a_{13}a_{21}+a_{11}a_{23}&a_{12}a_{22}&a_{13}a_{22}+a_{12}a_{23}&a_{13}a_{23}\\
      a_{11}a_{31}&a_{12}a_{31}+a_{11}a_{32}&a_{13}a_{31}+a_{11}a_{33}&a_{12}a_{32}&a_{13}a_{32}+a_{12}a_{33}&a_{13}a_{33}\\
      a_{21}^{2}&2\,a_{21}a_{22}&2\,a_{21}a_{23}&a_{22}^{2}&2\,a_{22}a_{23}&a_{23}^{2}\\
      a_{21}a_{31}&a_{22}a_{31}+a_{21}a_{32}&a_{23}a_{31}+a_{21}a_{33}&a_{22}a_{32}&a_{23}a_{32}+a_{22}a_{33}&a_{23}a_{33}\\
      a_{31}^{2}&2\,a_{31}a_{32}&2\,a_{31}a_{33}&a_{32}^{2}&2\,a_{32}a_{33}&a_{33}^{2}
\end{pmatrix}\,.
\end{aligned}$}
\end{equation}
For the code used, see the script \verb|symmetric_power_matrix.m2| available at \cite{salizzoni2025supplementary}.\hfill$\diamondsuit$
\end{example}

\begin{definition}\label{def: nonlinear Kalman matrix}
Let $d\ge 1$ and $n\ge 0$ be integer numbers. Recall that $N\coloneqq\binom{n-1+d}{d}$. Let $A=(a_{ij})$ be an $n\times n$ matrix filled with variables and fix $p\le N$. Given a matrix $C\in\C^{p\times N}$, the {\em Kalman matrix of order $d$} of $C$ is the matrix
\begin{equation}\label{eq: nonlinear Kalman matrix}
    K_d(C) = 
    \begin{pmatrix}
        C\\
        C\rho_d(A)\\
        C\rho_d(A)^2\\
        \vdots\\
        C\rho_d(A)^{N-p}
    \end{pmatrix}
\end{equation}
with $p\left[N-p+1\right]$ rows and $N$ columns filled with entries in $\C[A]$.
\end{definition}

The following result is an almost immediate partial generalization of \cite[Prop. 1.1]{ottaviani2013matrices}.

\begin{proposition}\label{prop: containment Kalman variety determinantal variety of Kalman matrix}
Let $X\subseteq\PP^{n-1}$ be a variety and suppose that its vanishing ideal is minimally generated by the homogeneous polynomials $f_1,\ldots,f_p$ of degrees $d_1\le\cdots\le d_p$ respectively. Set $d \coloneqq \mathrm{lcm}(d_1,\ldots,d_p)$. Let $r_1,\ldots,r_p$ be the row vectors of coefficients of the polynomials
$f_1^{d/d_1},\ldots,f_p^{d/d_p}$. Let $C_X$ be the matrix whose rows are a basis for the vector space $\langle r_1,\ldots,r_p\rangle\subseteq\C^N$.
Then the Kalman variety $\kk(X)$ is contained in the variety cut out by the maximal minors of the Kalman matrix $K_d(C_X)$.
\end{proposition}
\begin{proof}
The hypotheses tell us that $X=\nu_d^{-1}(\nu_d(\PP^{n-1})\cap\PP(\ker(C_X)))$. In particular, given a point $x=[v]\in\PP^{n-1}$, then $f_1(x)=\cdots=f_p(x)=0$ if and only if $C_X\,\nu_d(v)=0$. Furthermore, for every $[A]\in\kk(X)$, if $x=[v]\in X$ is an eigenpoint of $A$, then $\nu_d(x)=[\nu_d(v)]\in\PP(\ker(C_X))$ is an eigenpoint of the symmetric power $\rho_d(A)$. By applying \cite[Prop. 1.1]{ottaviani2013matrices}, we conclude that all maximal minors of $K_d(C_X)$ must vanish at $A$.
\end{proof}

In the following example and in the upcoming sections, we often consider a hypersurface $X\subseteq\PP^{n-1}$ cut out by a polynomial $f\in\C[x]_d$. In this case, we use the shorthand $K_d(f)$ to denote the Kalman matrix $K_d(C_f)$, where $C_f$ is the row vector of coefficients of $f$.

\begin{example}\label{ex: plane conic Kalman matrix}
Consider the plane conic $Q\subseteq\PP^2$ defined by the polynomial $f$ in \eqref{eq: plane conic with parameters}.
We have $Q=\{x=[v]\in\PP^2 \mid C_f\,\nu_2(v)=0\}$, where
\[
C_f =
\begin{pmatrix}
    b_{200} & b_{110} & b_{101} & b_{020} & b_{011} & b_{002}
\end{pmatrix}\,.
\]
As explained above, the conic $Q$ is isomorphic, via the Veronese embedding, to the linear section of the Veronese surface $\nu_2(\PP^2)$ with the hyperplane $\PP(\ker(C_f))$.
Given $A\in\C^{3\times 3}$, consider its symmetric power $\rho_2(A)\in\C^{6\times 6}$ written explicitly in \eqref{eq: symmetric power 3x3 matrix}.
The Kalman matrix $K_2(f)$ is
\[
K_2(f) =
\begin{pmatrix}
    C_f\\
    C_f\rho_d(A)\\
    C_f\rho_d(A)^2\\
    \vdots\\
    C_f\rho_d(A)^5
\end{pmatrix}
\in\C^{6\times 6}\,.
\]
The determinant $\det K_2(f)$ is a homogeneous polynomial in $\C[A]$ of degree $\sum_{i=0}^52i=2\binom{6}{2}=30$. Recall that the nonlinear Kalman variety $\kk(Q)$ is a hypersurface of degree $6$ in $\PP(\mathrm{Sym}^2\C^3)$. By Proposition~\ref{prop: containment Kalman variety determinantal variety of Kalman matrix}, its equation must divide $\det K_2(f)$. More precisely, we have
\begin{equation}\label{eq: factorization det Kalman conic}
\det K_2(f) = (\det A)^3\cdot g\cdot p\cdot q_1\cdot q_2\,,
\end{equation}
where $\deg g=3$ and $\deg p=\deg q_1=\deg q_2=6$. The polynomials $g$ and $p$ do not depend on $Q$, while $q_1$ and $q_2$ depend on $Q$. Indeed, one of the two polynomials $q_i$ is the equation of $\kk(Q)$. We study this factorization in more detail in Section \ref{sec: Kalman hypersurfaces}.
For the code used, see the script \verb|nonlinear_Kalman_matrix_conic.m2| available at \cite{salizzoni2025supplementary}.\hfill$\diamondsuit$
\end{example}

As we have seen in Example~\ref{ex: plane conic Kalman matrix}, the variety defined by the maximal minors of the Kalman matrix may be larger than the Kalman variety. In fact, when $X$ is a hypersurface of degree $d>1$, then the Kalman matrix is square, and its determinant has a factor that is independent of the defining equation of $X$ (compare with the factors $g$ and $p$ in \eqref{eq: factorization det Kalman conic}). In particular, in this case, we do not get equality between the determinant of the Kalman matrix and the defining equation of the Kalman variety. We focus on this fact in the upcoming Section~\ref{sec: Kalman hypersurfaces}.

\section{Kalman varieties of hypersurfaces}\label{sec: Kalman hypersurfaces}

In this section, we focus on the case when $X\subseteq\PP^{n-1}$ is a hypersurface cut out by a homogeneous polynomial $f\in\C[x]_d$. Recall Notation \ref{notations} and the Kalman matrix of order $d$ of $C_f$ in~\eqref{eq: nonlinear Kalman matrix}, which we write simply as $K_d(f)$.
In this case, it is a square matrix of size $N$ filled with polynomials in $\C[A]$.
We wish to describe the factors of the polynomial
\[
\det K_d(f)=\det K_d(f)(A)\in\C[A]\,.
\]
One of the irreducible factors of $\det K_d(f)$ is the equation of $\kk(X)$. The main result of this section (Theorem~\ref{thm: factors determinant Kalman matrix}) describes all irreducible factors of $\det K_d(f)$ and their multiplicities. This allows us to write the equation of $\kk(X)$ as in \eqref{eq: wannabe ratio}.

We introduce some further notations needed thereafter. Let $n$ and $d$ be positive integers. A partition of $d$ is a sequence $\mu=(\mu_1,\ldots,\mu_s)$ of positive integers written in non-decreasing order whose sum is $d$. We use the notation $\mu\vdash d$ to denote a partition of $d$ and define
\[
    P_d^{\mle n}\coloneqq\{\mu=(\mu_1,\ldots,\mu_s)\vdash d\mid s\le n\}\,,\quad P_d\coloneqq P_d^{\mle d}\,.
\]
An asymptotic formula for the cardinality of $P_d$ was given in~\cite{hardy1918asymptotic}, while an exact formula appeared in~\cite{bruinier2013algebraic}. In general, we have $|P_d^{\mle n}|= \sum_{i=1}^{\min(n,d)}p_{d,i}$, where $p_{d,i}$ is the number of partitions of $d$ in exactly $i$ parts, and we can compute $p_{d,i}$ recursively using the fact that $p_{d,i}=p_{d-1,i-1}+p_{d-i,i}$. 

Given $f\in\C[x]_d$, for each $\mu\in P_d^{\mle n}$ we denote by $f_{\mu}$ the image of $f$ with respect to the inclusion $\mathrm{Sym}^d\C^n\hookrightarrow\bigotimes_{i=1}^s\mathrm{Sym}^{\mu_i}\C^n$, or equivalently $\C[x]_d\hookrightarrow\bigotimes_{i=1}^s\C[x^{(i)}]_{\mu_i}$ where $x^{(i)}=(x_1^{(i)},\ldots,x_n^{(i)})$ is a vector of variables for every $i\in[s]$. In other words, we read $f$ either as a homogeneous polynomial of degree $d$ or as an $s$-homogeneous polynomial $f_\mu$ of multidegree $\mu$.

\begin{example}\label{ex: f11}
Consider the quadratic form $f(x_1,x_2,x_3)=x_2^2-x_1x_3$ and the partition $\mu=(1,1)$ of $2$. The bilinear form $f_{(1,1)}\in\C[x_1^{(1)},x_2^{(1)},x_3^{(1)}]\otimes\C[x_1^{(2)},x_2^{(2)},x_3^{(2)}]$ associated to $f$ is
\[
    f_{(1,1)} = x_2^{(1)}x_2^{(2)}-\frac{1}{2}\left(x_1^{(1)}x_3^{(2)}+x_3^{(1)}x_1^{(2)}\right)
\]
and is equivalent to reading a $3\times 3$ symmetric matrix in the larger space of $3\times 3$ nonsymmetric matrices.\hfill$\diamondsuit$
\end{example}

\begin{definition}\label{def: mu Kalman}
Let $f\in\C[x]_d\setminus\{0\}$ and $\mu=(\mu_1,\ldots,\mu_s)\in P_d^{\mle n}$.
The {\em $\mu$-Kalman variety $\kk_{\mu}(f)$} is the Zariski closure of the set 
\[
\resizebox{\textwidth}{!}{
$\begin{aligned}
\{[A]\in\PP^{n^2-1} \mid\text{$\exists\,v_1,\ldots,v_s$ eigenvectors of $A$ such that $f_{\mu}(v_1,\ldots,v_s)=0$ and $\dim\langle v_1,\ldots,v_s\rangle=s$}\}\,.
\end{aligned}$}
\]
\end{definition}

Observe that, for the trivial partition $\mu=(d)$, then $\kk_{(d)}(f)$ is the Kalman variety of $\V(f)$. In the following, we simply write $\kk(f)$ for $\kk(\V(f))$. Our first goal is to prove the following result.

\begin{theorem}\label{thm: degrees generalized Kalman varieties}
Let $f\in\C[x]_d\setminus\{0\}$. For any partition $\mu=(\mu_1,\ldots,\mu_s)\in P_d^{\mle n}$, the $\mu$-Kalman variety $\kk_{\mu}(f)$ is a hypersurface of degree
\begin{equation}\label{eq: degrees generalized Kalman varieties}
\deg\kk_{\mu}(f) = \frac{d\binom{n}{2}(n-1)_{s-1}}{m_1!\cdots m_d!}=\frac{(n-1)d}{2}\binom{n}{n-s,m_1,\ldots,m_d}\,,
\end{equation}
where $(a)_b\coloneqq\prod_{j=0}^{b-1}(a-j)$ is the falling factorial (with $(a)_0\coloneqq 1$) and $m_i=|\{j\in[s]\mid\mu_j=i\}|$ for all $i\in[d]$.
Furthermore, if $f$ is irreducible, then also $\kk_{\mu}(f)$ is irreducible.
\end{theorem}

Observe that, for $\mu=(d)$, then \eqref{eq: degrees generalized Kalman varieties} simplifies to $\deg\kk_{(d)}(f)=\deg\kk(f)=d\binom{n}{2}$ as in Proposition~\ref{prop: dim deg kalman}.

\begin{example}
One might choose one of the two equivalent expressions in \eqref{eq: degrees generalized Kalman varieties}. For example, assume $n=5$, $d=9$, and consider the partition $\mu=(1,2,2,4)$. Then $m_1=1$, $m_2=2$, $m_4=1$ and $m_i=0$ for all $i\in[9]\setminus\{1,2,4\}$. Furthermore $(4)_3=24$. Then
\[
\deg\kk_\mu(f)=\frac{9\cdot\binom{5}{2}\cdot 24}{2!} = 1080 = \frac{4\cdot 9}{2}\cdot\frac{5!}{2!}\,.\tag*{\text{$\diamondsuit$}}
\]
\end{example}

The proof of Theorem~\ref{thm: degrees generalized Kalman varieties} requires a few preliminary results.
In the following, we denote by $P=\{P_1,\ldots,P_k\}$ a partition of $[s]=\{1,\ldots,s\}$, where $\emptyset\neq P_i\subseteq[s]$ for all $i$.
Let $\kp_s$ denote the set of partitions of $[s]$.

\begin{definition}\label{def: P-compatible}
Given $P=\{P_1,\ldots,P_k\}\in\kp_s$, we say that an $s$-tuple $(v_1,\ldots,v_s)$ of nonzero vectors of $\C^n$ is {\em $P$-compatible} if $\dim\left\langle v_p\mid p\in P_i\right\rangle=1$ for all $i\in[k]$ and if $\dim\left\langle v_{j_1},\ldots,v_{j_k}\right\rangle=k$ where $j_i\in P_i$ for all $i\in[k]$.
\end{definition}

Let $n\ge s>0$ be two integers, and consider the variety
\begin{equation}
    W_s \coloneqq \left\{([A],[v_1],\ldots,[v_s])\mid\text{$v_1,\ldots,v_s$ are eigenvectors of $A$}\right\}\subseteq\PP_s\,,
\end{equation}
where $\PP_s\coloneqq\PP^{n^2-1}\times(\PP^{n-1})^{\times s}$. Furthermore, for each $P\in\kp_s$ consider the variety
\[
W_{s,P}^\circ \coloneqq \left\{([A],[v_1],\ldots,[v_s])\mid\text{$v_1,\ldots,v_s$ are eigenvectors of $A$ and $(v_1,\ldots,v_s)$ is $P$-compatible}\right\}
\]
and its Zariski closure $W_{s,P}\coloneqq\overline{W_{s,P}^\circ}$.

\begin{lemma}\label{lem: Ws decomposition}
For every $P\in\kp_s$, the variety $W_{s,P}$ is irreducible of codimension $s(n-1)$ in $\PP_s$.
Moreover, $\pi_1(\overline{W_s\setminus\bigcup_{P\in\kp_s}W_{s,P}})$ is the set of matrices with a two-dimensional eigenspace.
\end{lemma}
\begin{proof}
Write $P=\{P_1,\ldots,P_k\}$. Consider the projection $\pi_2\colon W_{s,P}\to(\PP^{n-1})^{\times s}$. Then $\image\pi_2=\overline{\ku_{s,P}}$, where
\[
\ku_{s,P} \coloneqq \left\{([v_1],\ldots,[v_s])\mid\text {$(v_1,\ldots,v_s)$ is $P$-compatible}\right\}\,.
\]
This means that, for every $([v_1],\ldots,[v_s])\in \ku_{s,P}$, we have $\dim\left\langle v_p\mid p\in P_i\right\rangle=1$ for all $i\in[k]$ and $\dim\left\langle v_{j_1},\ldots,v_{j_k}\right\rangle=k$ where $j_i\in P_i$ for all $i\in[k]$.
Hence, the codimension of the image of $\pi_2$ is the sum of the codimensions of the determinantal varieties of $|P_i|\times n$ matrices of rank at most one, that is $\sum_{i=1}^k(|P_i|-1)(n-1)=(s-k)(n-1)$.
Additionally, the image of $\pi_2$ is isomorphic to $(\PP^{n-1})^{\times k}$ because $P$ is a partition, in particular $\image\pi_2$ is irreducible.
Now pick an element $([v_1],\ldots,[v_s])\in \ku_{s,P}$ and choose a representative $v_{j_i}$ with $j_i\in P_i$ for all $i\in[k]$.
We need to impose that $\rank(Av_{j_i}\mid v_{j_i})\le 1$ for every $i\in[k]$, which gives in total $k(n-1)$ linearly independent conditions on the matrix $A$. This means that all fibers of $\pi_2$ over $\ku_{s,P}$ are linear subspaces of codimension $k(n-1)$, and $W_{s,P}^\circ$ is the total space of a vector bundle over $\ku_{s,P}$. Summing up, the incidence variety $W_{s,P}$ is irreducible, and by the theorem on the dimension of fibers, its codimension is $\codim W_{s,P}=(s-k)(n-1)+k(n-1)=s(n-1)$.

For the last part of the statement, observe first that $W_s\supseteq\bigcup_{P\in\kp_s}W_{s,P}$ by definition. Consider a point $([A],[v_1],\ldots,[v_s])\in W_s\setminus\bigcup_{P\in\kp_s}W_{s,P}$. Then, there exist $m$ indices $j_1,\ldots,j_{m}$ such that $[v_{j_1}],\ldots,[v_{j_m}]$ are pairwise distinct and $\dim\langle v_{j_1},\ldots,v_{j_{m}}\rangle=m-1$. This implies that $A$ has an eigenspace of dimension at least $2$.
\end{proof}

\begin{example}\label{ex: decomposition W_3}
In this example, we compute the irreducible decomposition of $W_3$ for $n=3$ using \verb|Macaulay2|. Applying Lemma~\ref{lem: Ws decomposition}, we have that $W_{3,P}$ is irreducible of codimension $6$ in $\PP_3=\PP^8\times(\PP^2)^{\times 3}$ for all $P\in\kp_s$, namely for $P=\{\{1\},\{2\},\{3\}\}$, $P=\{\{1,2\},\{3\}\}$, $P=\{\{1,3\},\{2\}\}$, $P=\{\{1\},\{2,3\}\}$, and $P=\{\{1,2,3\}\}$. Setting $E_3\coloneqq\overline{W_3\setminus\bigcup_{P\in\kp_3}W_{3,P}}$, we verified that $E_3$ also has codimension $6$ in $\PP_3$ and
\[
W_3 = W_{3,\{\{1\},\{2\},\{3\}\}}\cup W_{3,\{\{1,2\},\{3\}\}}\cup W_{3,\{\{1,3\},\{2\}\}}\cup W_{3,\{\{1\},\{2,3\}\}}\cup W_{3,\{\{1,2,3\}\}}\cup E_3
\]
is an irreducible decomposition of $W_3$ over $\Q$. The projection of $E_3$ onto $\PP^8$ is the variety of points $[A]$ such that $A\in\C^{3\times 3}$ has a two-dimensional eigenspace.
For the code used, see the script \verb|variety_W_3.m2| available at \cite{salizzoni2025supplementary}.\hfill$\diamondsuit$
\end{example}

For every $s$ we define $\widetilde{W}_s \coloneqq W_{s,P}$ when $P=\{\{1\},\ldots,\{s\}\}$. The previous result tells us that $\widetilde{W}_s$ is an irreducible variety in $X$. We are interested in its class $[\widetilde{W}_s]\in A^{s(n-1)}(\PP_s)$, where
\[
A^*(\PP_s)=\bigoplus_{i\ge 0}A^i(\PP_s)=\frac{\Z[h_0,h_1,\ldots,h_s]}{(h_0^{n^2},h_1^n,\ldots,h_s^n)}
\]
is the Chow ring of $\PP_s$ and $h_i$ is the pullback of a hyperplane section from the $i$-th factor in $\PP_s$, for all $i\in[s]$.

Computing the complete class $[\widetilde{W}_s]$ is, in general, a hard task, and we leave its description to research further.
As explained later, here we are interested only in the linear part of $[\widetilde{W}_s]$ in the variable $h_0$.
In the following lemma, we provide a recursive formula to compute every class $[W_{s,P}]$ using the classes $[\widetilde{W}_s]$.

\begin{lemma}\label{lem: multidegrees equality}
Consider a partition $P=\{P_1,\ldots,P_k\}\in\kp_s$.
Let $q_i\coloneqq\min P_i$ for all $i\in[k]$. Consider the classes $[\widetilde{W}_k]\in A^{k(n-1)}(\PP_k)$ and $[W_{s,P}]\in^{s(n-1)}(\PP_s)$ where 
\[
A^*(\PP_k)=\frac{\Z[h_0,\ldots,h_k]}{(h_0^{n^2},h_1^n\ldots,h_k^n)}\text{ and }A^*(\PP_s)=\frac{\Z[t_0,\ldots,t_s]}{(t_0^{n^2},t_1^n,\ldots,t_s^n)}\,.
\]
Define the map
\[
    \varphi_P\colon A^*(\PP_k)\to A^*(\PP_s)\,,\quad
    \varphi_P(h_i)\coloneqq
    \begin{cases}
        t_0 & \text{if $i=0$,}\\
        t_{q_i} & \text{if $i\in[k]$.}
    \end{cases}
\]
Then
\[
[W_{s,P}] = \varphi_P([\widetilde{W}_k])\cdot\prod_{i=1}^k\prod_{p\in P_i\setminus\{q_i\}}\sum_{r=0}^{n-1}t_{q_i}^{n-1-r}t_{p}^r\,,
\]
where the product over $p\in P_i\setminus\{q_i\}$ is set to $1$ when $P_i$ has only one element, namely $P_i=\{q_i\}$.
\end{lemma}
\begin{proof}
Recall that we write $\PP_\ell=\PP^{n^2-1}\times \prod_{i=1}^\ell\PP^{n-1}$ for any positive integer $\ell$. Consider the projection map
\[
\pi_P\colon \PP_s\to\PP_k\,,\quad\pi_P(A,v_1,\ldots,v_s)\coloneqq(A,v_{q_1},\ldots,v_{q_k})\,.
\] 
The pullback along $\pi_P$ induces the inclusion $\varphi_P\colon A^*(\PP_k)\to A^*(\PP_s)$ defined above. For the pullback of $[\widetilde{W}_k]$ along $\pi_P$ we find that
\[
    \varphi_P([\widetilde{W}_k])=\left[\widetilde{W}_k\times\prod_{i=1}^k\prod_{p\in P_i\setminus\{q_i\}}\PP^{n-1}\right]
\]
By definition, $W_{s,P}$ is the subvariety of $\widetilde{W}_k\times\prod_{i=1}^k\prod_{p\in P_i\setminus\{q_i\}}\PP^{n-1}$ consisting of all points $(A,v_1,\ldots,v_s)$ where for each part $P_i$ of the partition $P$, all $2\times 2$ minors of the matrix $B_i\in\C^{n\times |P_i|}$ whose columns are $v_j$ for all $j\in P_i$, vanish.
Applying \cite[Exerc. 15.5(b)]{miller2005combinatorial} (see also \cite[Cor. 16.27]{miller2005combinatorial}) and pulling back to $\PP_s$, we find that for $i$ as above, the class corresponding to the subvariety cut out by all $2\times 2$ minors of $B_i$ is the complete homogeneous symmetric function of degree $n-1$ in $|P_i|$ variables
\[
    \sigma_i\coloneqq\sum_{\substack{\alpha\in \N_0^{|P_i|} \\|\alpha|=n-1}}\prod_{p\in P_i}t_p^{\alpha_p}\,.
\]
First notice that the intersection of $\widetilde{W}_k\times\prod_{i=1}^k\prod_{p\in P_i\setminus\{q_i\}}\PP^{n-1}$ with the vanishing locus of the ideal of $2\times 2$ minors of $B_i$ is transversal since the coordinates $v_p$ with $p\in P_i\setminus \{q_i\}$ are unconstrained in $\widetilde{W}_k\times\prod_{i=1}^k\prod_{p\in P_i\setminus\{q_i\}}\PP^{n-1}$ and the intersection forces these coordinates to be equal to $v_{q_i}$. The intersection stays transversal even when we do the intersection for all values of $i\in[k]$ at once, since the coordinates that get constrained by the conditions that we impose for two distinct values of $i$ are disjoint. In other words, we get
\[
[W_{s,P}] = \left[\widetilde{W}_k\times\prod_{i=1}^k\prod_{p\in P_i\setminus\{q_i\}}\PP^{n-1}\right]\cdot\prod_{i=1}^k \sigma_i = \varphi_P([\widetilde{W}_k])\cdot\prod_{i=1}^k \sigma_i\,.
\]
To finish the proof we note that, modulo the conditions $t_p^{n}=0$ for all $p\in [s]$, each class $\sigma_i$ factors as
\[
    \sigma_i=\prod_{p\in P_i\setminus\{q_i\}}\sum_{r=0}^{n-1}t_{q_i}^{n-1-r}t_p^r\,.
\]
Geometrically, this corresponds to the fact that the vanishing of all $2\times 2$ minors of $B_i$ is equivalent to the vanishing of all $2\times 2$ minors that include the first column.
\end{proof}

\begin{example}\label{ex: apply equalities multidegrees}
We apply Lemma~\ref{lem: multidegrees equality} for small values of $s$. For simplicity, we adopt the following notations:
\[
A^*(\PP_1)=\frac{\Z[h_0,h_1]}{(h_0^{n^2},h_1^n)}\,,\quad A^*(\PP_2)=\frac{\Z[g_0,g_1,g_2]}{(g_0^{n^2},g_1^n,g_2^n)}\,,\quad A^*(\PP_3)=\frac{\Z[t_0,t_1,t_2,t_3]}{(t_0^{n^2},t_1^n,t_2^n,t_3^n)}\,.
\]
When $s=1$, we know already that
\[
    [\widetilde{W}_1]=[W_1]=[W_{1,\{\{1\}\}}]=\sum_{j=0}^{n-1}\binom{n}{j}h_0^{n-1-j}h_1^j\,.
\]
Consider $s=2$. Then
\begin{align*}
    [W_{2,\{\{1,2\}\}}] &= \varphi_{\{\{1,2\}\}}([\widetilde{W}_1])\cdot\sum_{r=0}^{n-1}g_1^{n-1-r}g_2^r = \sum_{j=0}^{n-1}\binom{n}{j}g_0^{n-1-j}g_1^j\cdot\sum_{r=0}^{n-1}g_1^{n-1-r}g_2^r\,,
\end{align*}
where $\varphi_{\{\{1,2\}\}}(h_0)=g_0$ and $\varphi_{\{\{1,2\}\}}(h_1)=g_1$. In this case $W_2=\widetilde{W}_2\cup W_{2,\{\{1,2\}\}}$ is an irreducible decomposition of $W_2$ and $\codim \widetilde{W}_2=\codim W_{2,\{\{1,2\}\}}=2(n-1)$ by Lemma~\ref{lem: Ws decomposition}. We can use this information to conclude that
\begin{align*}
    [\widetilde{W}_2] &= [W_2]-[W_{2,\{\{1,2\}\}}]\\
    &= \prod_{i=1}^2\left(\sum_{j=0}^{n-1}\binom{n}{j}g_0^{n-1-j}g_i^j\right)-\sum_{j=0}^{n-1}\binom{n}{j}g_0^{n-1-j}g_1^j\cdot\sum_{r=0}^{n-1}g_1^{n-1-r}g_2^r\\
    &= \sum_{j=0}^{n-1}\binom{n}{j}g_0^{n-1-j}g_1^j\cdot\sum_{r=0}^{n-1}\left(\binom{n}{r}g_0^{n-1-r}-g_1^{n-1-r}\right)g_2^r\,.
\end{align*}
Now consider $s=3$ and the partitions $\{\{1,2,3\}\}$, $\{\{i,j\},\{k\}\}$ for all three ordered tuples $(i,j,k)$ such that $\{i,j,k\}=[3]$. Then
\begin{align*}
\varphi_{\{\{1,2,3\}\}}(h_0) &= t_0\,,\ \varphi_{\{\{1,2,3\}\}}(h_1)=t_1\\
\varphi_{\{\{i,j\},\{k\}\}}(g_0) &= t_0\,,\ \varphi_{\{\{i,j\},\{k\}\}}(g_1) = t_i\,,\ \varphi_{\{\{i,j\},\{k\}\}}(g_2) = t_k\,.
\end{align*}
Using this information, we get
\begin{align*}
    [W_{3,\{\{1,2,3\}\}}] &= \varphi_{\{\{1,2,3\}\}}([\widetilde{W}_1])\cdot\prod_{p\in\{2,3\}}\sum_{r=0}^{n-1}t_1^{n-1-r}t_p^r = \sum_{j=0}^{n-1}\binom{n}{j}t_0^{n-1-j}t_1^j\prod_{p\in\{2,3\}}\sum_{r=0}^{n-1}t_1^{n-1-r}t_p^r\\
    [W_{3,\{\{i,j\},\{k\}\}}] &= \varphi_{\{\{i,j\},\{k\}\}}([\widetilde{W}_2])\cdot\sum_{r=0}^{n-1}t_i^{n-1-r}t_j^r\\
    &= \sum_{j=0}^{n-1}\binom{n}{j}t_0^{n-1-j}t_k^j\cdot\sum_{r=0}^{n-1}\left(\binom{n}{r}t_0^{n-1-r}-t_k^{n-1-r}\right)t_i^r\cdot\sum_{r=0}^{n-1}t_i^{n-1-r}t_j^r\,,
\end{align*}
which we use in Example \ref{ex: multidegrees components W_3}.\hfill$\diamondsuit$
\end{example}

Using the Lemma~\ref{lem: multidegrees equality}, we can determine recursively the linear part of every class $[W_{s,P}]$ in the variable $h_0$ using the classes $[\widetilde{W}_s]$. In the following statement, every time we consider a sum of monomials $g\in A^*(\PP_s)$ and the class $[Y]\in A^*(\PP_s)$ for some subvariety $Y\subseteq \PP_s$, we say that the coefficient of $g$ in $[Y]$ is $c$ if every monomial of $g$ appears in $[Y]$ with coefficient $c$.

\begin{corollary}\label{corol: coefficient equality}
    For every partition $P\in\kp_s$, the coefficient $c_{s,P}$ of $h_0\sum_{i=1}^s h_1^{n-1}\cdots h_i^{n-2}\cdots h_s^{n-1}$ in $[W_{s,P}]\in A^{s(n-1)}(\PP_s)$ is equal to the coefficient $\tilde{c}_s$ of $h_0\sum_{i=1}^{|P|} h_1^{n-1}\cdots h_i^{n-2}\cdots h_{\mcp}^{n-1}$ in $[\widetilde{W}_{\mcp}]\in A^{|P|(n-1)}(\PP_{\mcp})$.
\end{corollary}

In the following lemma, we compute the coefficient $c_{s,P}$ defined in Corollary~\ref{corol: coefficient equality} explicitly.

\begin{lemma}\label{lem: coefficient of W tilde}
    Consider the notations of Lemma~\ref{lem: multidegrees equality} and Corollary~\ref{corol: coefficient equality}.
    For every partition $P\in\kp_s$, the coefficient $c_{s,P}$ of $h_0\sum_{i=1}^s h_1^{n-1}\cdots h_i^{n-2}\cdots h_s^{n-1}$ in $[W_{s,P}]\in A^{s(n-1)}(X)$ is equal to $\binom{n}{2}(n-1)_{|P|-1}$.
    In particular, the coefficient $\tilde{c}_s$ of $h_0\sum_{i=1}^s h_1^{n-1}\cdots h_i^{n-2}\cdots h_s^{n-1}$ in $[\widetilde{W}_s]\in A^{s(n-1)}(X)$ is equal to $\binom{n}{2}(n-1)_{s-1}$.
\end{lemma}
\begin{proof}
First, we notice that $[W_s] = \prod_{i=1}^s\sum_{j=0}^{n-1}\binom{n}{j}h_0^{n-1-j}h_i^j$, in particular the coefficient of $h_0\sum_{i=1}^s h_1^{n-1}\cdots h_i^{n-2}\cdots h_s^{n-1}$ is equal to $\binom{n}{2}n^{s-1}$. To prove the statement, we proceed by induction on $s$. When $s=1$, then $\widetilde{W}_1=W_1$ and so $\tilde{c}_1=\binom{n}{2}$. We now assume that $\tilde{c}_k=\binom{n}{2}(n-1)_{k-1}$ for $k\le s-1$ and we prove that this equality holds for $k=s$.  By Lemma~\ref{lem: Ws decomposition} we have
\begin{equation}\label{eq: recursive formula}
    [\widetilde{W}_s]=[W_s]-\sum_{\substack{P\in\kp_s\\P\neq\{\{1\},\ldots,\{s\}\}}}[W_{s,P}]-[Y]
\end{equation}
where $Y:=\overline{W_s\setminus\bigcup_{P\in\kp_s}W_{s,P}}$. By Lemma~\ref{lem: Ws decomposition} we know that $\pi_1(Y)$ is the set of matrices with an eigenspace of dimension at least two in $\PP^{n^2-1}$. In particular $\codim\pi_1(Y)=3$, see~\cite[Tab. 2]{keller2008multiple}. This implies that the coefficient of $h_0\sum_{i=1}^s h_1^{n-1}\cdots h_i^{n-2}\cdots h_s^{n-1}$ in the class $[Y]$ is zero. From \eqref{eq: recursive formula} we obtain
\[
    \tilde{c}_s=\binom{n}{2}n^{s-1}-\sum_{\substack{P\in\kp_s\\P\neq\{\{1\},\ldots,\{s\}\}}}c_{s,P}\,.
\]
Therefore, by Corollary~\ref{corol: coefficient equality}, the coefficient of $h_0\sum_{i=1}^s h_1^{n-1}\cdots h_i^{n-2}\cdots h_s^{n-1}$ in the class $[\widetilde{W}_s]$ is
\begin{align*}
\tilde{c}_s&=\binom{n}{2}\left(n^{s-1}-\sum_{\substack{P\in\kp_s\\P\neq\{\{1\},\ldots,\{s\}\}}}(n-1)_{|P|-1}\right)= \binom{n}{2}\left( n^{s-1}-\sum_{k=1}^{s-1}{\genfrac\{\}{0pt}{0}{s}{k}}\binom{n-1}{k-1}(k-1)!\right)\\
&=\binom{n}{2}\left(n^{s-1}-\frac{1}{n}\sum_{k=0}^{s-1}{\genfrac\{\}{0pt}{0}{s}{k}}\binom{n}{k}k!\right)=\binom{n}{2}\left(n^{s-1}-\frac{1}{n}\sum_{k=0}^{s}{\genfrac\{\}{0pt}{0}{s}{k}}\binom{n}{k}k!+\frac{1}{n}\binom{n}{s}s!\right)\\
&\stackrel{(\star)}{=}\binom{n}{2}\binom{n-1}{s-1}(s-1)!=\binom{n}{2}(n-1)_{s-1}\,,
\end{align*}
where for every $k\le s$, the number $\genfrac\{\}{0pt}{1}{s}{k}$ is the $(s,k)$-Stirling number of second type (or the number of partitions of $[s]$ in exactly $k$ disjoint sets) and the equality $(\star)$ follows from~\cite[Eq. (6.10)]{graham1994concrete}.
\end{proof}

\begin{example}\label{ex: multidegrees components W_3}
Consider Examples \ref{ex: decomposition W_3} and \ref{ex: apply equalities multidegrees}. The multidegrees of the six irreducible components of $W_3$ for $n=3$ are respectively
\[
\resizebox{\textwidth}{!}{
$\begin{aligned}
    [\widetilde{W}_3] &= 6\,e_3^2+6\,e_2e_3\,t_0+2(e_2^2+2\,e_1e_3)t_0^2+3(e_1e_2+e_3)t_0^3+(e_1^2+3\,e_2)t_0^4+2\,e_1t_0^5+t_0^6\\
    [W_{3,\{\{i,j\},\{k\}\}}] &= 6\,e_3^2+6\,e_2e_3t_0+2(b_2^2+4\,b_1b_2t_k+(b_1^2-b_2)t_k^2)t_0^2+3(b_1b_2+(b_1^2-b_2)t_k)t_0^3+(b_1^2-b_2)t_0^4\\
    [W_{3,\{\{1,2,3\}\}}] &= 3\,e_3^2+3\,e_2e_3t_0+(e_2^2-e_1e_3)t_0^2\\
    [E_3] &= 6\,e_3t_0^3+3\,e_2t_0^4+e_1t_0^5\,,
\end{aligned}$}
\]
where $e_\ell(t_1,t_2,t_3)$ is the $\ell$th elementary symmetric function for all $\ell\in[3]$, and in the expansion of $[W_{3,\{\{i,j\},\{k\}\}}]$, we consider all three ordered tuples $(i,j,k)$ such that $\{i,j,k\}=[3]$, while $b_\ell(t_i,t_j)$ is also the $\ell$th elementary symmetric function in $(t_i,t_j)$ for all $\ell\in[2]$. The first three multidegrees descend by Example \ref{ex: apply equalities multidegrees}, while the class $[E_3]$ can be derived from the script \verb|variety_W_3.m2| available at \cite{salizzoni2025supplementary}.
Observe also that the coefficient $\tilde{c}_3$ of the linear part in $h_0$ in $[\widetilde{W}_3]\in A^6(\PP_3)$ is equal to $6=\binom{3}{2}(2)_2$, as predicted by Lemma~\ref{lem: coefficient of W tilde}.\hfill$\diamondsuit$
\end{example}

We are now ready to prove the first main result of this section.

\begin{proof}[Proof of Theorem~\ref{thm: degrees generalized Kalman varieties}]
Consider a partition $\mu=(\mu_1,\ldots,\mu_s)\vdash d$. The fact that $\kk_{\mu}(f)$ has codimension one and is irreducible if $f$ is, follows by a similar argument to that in Proposition \ref{prop: dim deg kalman} by realising $\kk_{\mu}(f)$ as the image of a vector bundle over $\V(f_\mu)\subseteq(\PP^{n-1})^{\times s}$ under a projection. We focus on the degree computation.
The $\mu$-Kalman variety $\kk_{\mu}(f)$ is the projection of $\widetilde{W}_s\cap(\PP^{n^2-1}\times\V(f_\mu))$ onto $\PP^{n^2-1}$, where
\[
\V(f_\mu)=\{([v_1],\ldots,[v_s])\mid f_\mu(v_1,\ldots,v_s)=0\}\,.
\]
The polynomial $f$ and the partition $\mu$ yield a divisor $\V(f_\mu)$ in $(\PP^{n-1})^{\times s}$.
Our goal is showing that the subvarieties $\widetilde{W}_s$ and $\PP^{n^2-1}\times\V(f_\mu)$ of $\PP_s$ are generically transverse.
Consider a point $([A],x_1,\ldots,x_s)\in\widetilde{W}_s\cap(\PP^{n^2-1}\times\V(f_\mu))\subseteq\PP_s$ such that $(x_1,\ldots,x_s)$ is a nonsingular point of $\widetilde{W}_s$ and of $\PP^{n^2-1}\times\V(f_\mu)$.
We consider all previous varieties in their corresponding Segre embedding in $\PP(\C^{n\times n}\otimes(\C^n)^{\otimes s})$, in particular, we identify $([A],x_1,\ldots,x_s)$ with $[A\otimes x_1\otimes\cdots\otimes x_s]$. We also denote by $T_pZ$ the affine tangent space to $Z\subseteq\PP^{n-1}$ at a nonsingular point $p$ of $Z$, namely the cone over the projective tangent space of $Z$ at $p$.
On the one hand, we have that
\[
T_{([A],x_1,\ldots,x_s)}(\PP^{n^2-1}\times\V(f_\mu)) = A\otimes T_{(x_1,\ldots,x_s)}\V(f_\mu)+\C^{n\times n}\otimes x_1\otimes\cdots\otimes x_s\,.
\]
Furthermore, observe that $\pi_2(\widetilde{W}_s)=(\PP^{n-1})^{\times s}$ by Lemma~\ref{lem: Ws decomposition}. This tells us that
\[
A\otimes T_{(x_1,\ldots,x_s)}(\PP^{n-1})^{\times s}\subseteq T_{([A],x_1,\ldots,x_s)}\widetilde{W}_s\,.
\]
Summing up, we have
\begin{align*}
    T_{([A],x_1,\ldots,x_s)}\PP_s &= A\otimes T_{(x_1,\ldots,x_s)}(\PP^{n-1})^{\times s}+\C^{n\times n}\otimes x_1\otimes\cdots\otimes x_s\\
    &\subseteq T_{([A],x_1,\ldots,x_s)}\widetilde{W}_s+T_{([A],x_1,\ldots,x_s)}(\PP^{n^2-1}\times\V(f_\mu))\\
    &\subseteq T_{([A],x_1,\ldots,x_s)}\PP_s\,,
\end{align*}
hence the previous containments are equalities. This proves that the subvarieties $\widetilde{W}_s$ and $\PP^{n^2-1}\times\V(f_\mu)$ of $\PP_s$ are generically transverse, yielding
\[
[\widetilde{W}_s\cap(\PP^{n^2-1}\times\V(f_\mu))] = [\widetilde{W}_s]\cdot[\PP^{n^2-1}\times\V(f_\mu)] = [\widetilde{W}_s]\cdot\sum_{i=1}^s\mu_i h_i\,.
\]
It follows that the product between $\deg\kk_{\mu}(f)$ and the degree of the projection onto $\PP^{n^2-1}$ is equal to the coefficient of $h_0\prod_{i=1}^s h_i^{n-1}$ in the previous expansion, which is equal to the coefficient of $h_0\prod_{i=1}^s h_i^{n-1}$ in the product
\[
\tilde{c}_s\,h_0\sum_{i=1}^s h_1^{n-1}\cdots h_i^{n-2}\cdots h_s^{n-1}\cdot\sum_{i=1}^s\mu_i h_i = \tilde{c}_s(\mu_1+\cdots+\mu_s)\,h_0\prod_{i=1}^s h_i^{n-1} = \tilde{c}_s d\,h_0\prod_{i=1}^s h_i^{n-1}\,,
\]
where $\tilde{c}_s=\binom{n}{2}(n-1)_{s-1}$ by Lemma~\ref{lem: coefficient of W tilde}. Since the degree of the projection is equal to $m_1!\cdots m_d!$, the statement follows.
\end{proof}

\begin{corollary}\label{corol: factors of det Kalman matrix}
Let $f\in\C[x]_d$ and define the total Kalman variety of $f$ as
\[
    \kk_{\mathrm{tot}}(f)\coloneqq\bigcup_{\mu\in P_d^{\mle n}}\kk_{\mu}(f)\,.
\]
If $f$ is generic, the hypersurfaces $\kk_{\mu}(f)$ are the irreducible components of $\kk_{\mathrm{tot}}(f)$.
\end{corollary}

Recall the notations given at the beginning of the section.
After computing the degrees of all hypersurfaces $\kk_\mu(f)$, we now turn our attention to all irreducible factors of $\det K_d(f)$ and to their multiplicities. In particular, we compare $\det K_d(f)$ to the discriminant of the characteristic polynomial of $\rho_d(A)$, denoted by $\Delta_d\in\C[A]$, where $\rho_d(A)$ was introduced in Definition~\ref{def: symmetric power matrix}. We also denote with $\Delta\coloneqq\Delta_1\in\C[A]$ the discriminant of the characteristic polynomial of $A$.

Our first result in this direction is a geometric statement about the factors of $\det K_d(f)$ which do not depend on $C_f$.

\begin{proposition}\label{prop: discriminant_factors_without_multi}
Every irreducible factor of $\Delta_d$ is either a factor of $\det K_d(f)$ or of $\Delta$. In this way, all the factors of $\det K_d(f)$ that are independent of $f$ arise. 
\end{proposition}
\begin{proof}
By \cite[Prop. 1.1]{ottaviani2013matrices}, $\det K_d(f)(B)=0$ if and only if $\rho_d(B)$ has an eigenvector in the hyperplane $\PP(\ker(C_f))\subseteq\PP^{N-1}$. Hence, the irreducible factors of $\det K_d(f)$ which are independent of $C_f$ vanish precisely on those $B$ for which $\rho_d(B)$ has an eigenvector in every hyperplane of $\PP^{N-1}$. This condition is satisfied if and only if $\rho_d(B)$ has an eigenspace of dimension at least two. 

Let us investigate the eigenpairs of $\rho_d(B)$. If $(\lambda_1,v_1),\ldots,(\lambda_d,v_d)$ are eigenpairs of $B$ (not necessarily distinct), then $(\lambda_1\cdots\lambda_d,v_1\cdots v_d)$ is an eigenpair of $\rho_d(B)$. We claim that in this way we obtain a generating set for every eigenspace of $\rho_d(B)$. If $B$ is diagonalizable, this is obvious simply by a dimension count. Otherwise, one can approximate $B$ by a sequence of diagonal matrices and notice that the collapsing of any eigenspace of $B$ directly translates to the collapsing of several eigenspaces of $\rho_d(B)$. Again, by counting dimensions and the number of collapsing eigenspaces, we obtain the result.
From now on, let $\lambda_1,\ldots,\lambda_n$ be all the eigenvalues of $B$, appearing with their algebraic multiplicity. Then all degree $d$ monomials in $\lambda_1,\ldots,\lambda_n$ are all eigenvalues of $\rho_d(B)$ listed with their algebraic multiplicities.

We first prove that for any value of $n$, each irreducible factor of $\det K_d(f)$ which is independent of $C_f$ is a factor of $\Delta_d$. Let $q$ be such an irreducible factor of $\det K_d(f)$ and assume that $q(B)=0$ for some $B\in\C^{n\times n}$. It suffices to show that $\Delta_d(B)=0$. By the first paragraph, since $q(B)=0$, $\rho_d(B)$ must have an eigenspace of dimension at least two. This means that at least one eigenvalue of $\rho_d(B)$ has geometric multiplicity at least two. Since the geometric multiplicity lower bounds the algebraic multiplicity, this eigenvalue is a multiple root of the characteristic polynomial of $\rho_d(B)$ and hence $\Delta_d(B)=0$ by definition of the discriminant.

Next, let $q$ be an irreducible factor of $\Delta_d$ and assume first that there is no $\lambda\in\C$ such that $q=\lambda\,\Delta$. Let $B\in\C^{n\times n}$ be generic within the locus of all matrices satisfying $q(B)=0$. We wish to show $\det K_d(f)(B)=0$. To do so, we first argue that $B$ is diagonalizable. In fact, if $B$ were not diagonalizable, then $\Delta(B)=0$. Since by assumption $q$ is not a factor of $\Delta_d$, the latter defines a subvariety inside $\V(q)$ of codimension one. Hence, a generic matrix $B$ in $\V(q)$ does not have two equal eigenvalues, and therefore it must be diagonalizable. On the other hand, since $q(B)=0$, also $\Delta_d(B)=0$, hence there exists an eigenvalue of $\rho_d(B)$ with algebraic multiplicity at least two. Since $B$ and therefore also $\rho_d(B)$ is diagonalizable, algebraic and geometric multiplicities agree; in fact, $\rho_d(B)$ has an eigenspace of dimension at least two. This eigenspace is forced to meet the hyperplane defined by $C_f$, and any vector in this intersection belongs to the kernel of the matrix $K_d(C_f)$ evaluated at $B$. Hence $\det K_d(f)(B)=0$.

Finally, let us prove that, for a sufficiently generic $f$, the irreducible factors of $\Delta$ do not appear as factors of $\det K_d(f)$ but do appear as factors of $\Delta_d$. Clearly, if two eigenvalues of $B$ agree, then at least two eigenvalues of $\rho_d(B)$ agree, as those are monomials in the eigenvalues of $B$. Hence $\Delta$ divides $\Delta_d$. Assume by contradiction that $\Delta$ shares a factor $q$ with $\det K_d(f)$. Then a generic matrix $B$ satisfying $q(B)=0$ would also satisfy $\det K_d(f)(B)=0$. However, within the locus of matrices that have two eigenvalues that agree, the set of diagonalizable matrices is low-dimensional (this follows from the Jordan normal form, where diagonalizability requires an additional entry to be zero). Hence, a generic matrix $B$ satisfying $q(B)=0$ has one-dimensional eigenspaces that do not span the ambient space. But then also $\rho_d(B)$ has one-dimensional eigenspaces, and so by the first paragraph of this proof, the factors of $\det K_d(f)$ which are independent of $C_f$ do not vanish at $B$.
\end{proof}

All the vanishing loci of the factors of $\det K_d(f)$ are described by Corollary~\ref{corol: factors of det Kalman matrix} and Proposition~\ref{prop: discriminant_factors_without_multi}. It remains to compute their multiplicities. 

\begin{lemma}\label{lem: discriminant_factors_with_multi}
    Define $\Delta_d^{\sat}\in\C[A]$ as the polynomial of smallest total degree, unique up to a scalar, such that there exists a power $k\in\N_0$ with $\Delta_d=\Delta_d^{\sat}\cdot\Delta^k$. Then $\Delta_d^{\sat}$ is a perfect square in $\C[A]$ and $\sqrt{\Delta_d^{\sat}}$ divides $\det K_d(f)$.
\end{lemma}
\begin{proof}
We fix notation as follows: We write $A=(a_{ij})$, $\lambda=(\lambda_1,\ldots,\lambda_{n})$ for the $n$ distinct eigenvalues of $A$, which live in a finite field extension $F$ of $\C(A)\coloneqq\C(a_{ij}\mid 1\le i,j\le n)$ and we consider the matrix $T\in F^{n\times n}$ whose columns are a basis of eigenvectors of $A$. Note that $T$ is invertible over $F$, but specializing the variables $a_{ij}$ can lead to $T$ being singular. There is a dense open set in $\C^{n\times n}$ on which this does not happen (namely, the set of diagonalizable matrices).

We start by proving that $\Delta_d^{\sat}$ is a perfect square. Recall that $N=\binom{n-1+d}{d}$ is the number of monomials of degree $d$ in $\lambda_1,\ldots,\lambda_n$ and write
\[
    \Delta_d = \prod_{\substack{\alpha,\beta\in \N_0^n\\ |\alpha|=|\beta|=d}}(\lambda^\alpha-\lambda^\beta)=(-1)^{\binom{N}{2}}\prod_{\substack{\alpha,\beta\in \N_0^n\\ |\alpha|=|\beta|=d \\ \alpha \prec_{\text{lex}}\beta}}(\lambda^\alpha-\lambda^\beta)^2\,.
\]
Hence $\Delta_d$ is a perfect square over $\C[\lambda_1,\ldots,\lambda_n]$ (but usually not over $\C[A]$). We set
\[
    \tilde{\Delta}_d'\coloneqq\prod_{\substack{\alpha,\beta\in \N_0^n\\ |\alpha|=|\beta|=d \\ \alpha \prec_{\text{lex}}\beta}}(\lambda^\alpha-\lambda^\beta)
\]
such that $\tilde{\Delta}_d^2=\pm\Delta_d$.
Consider the symmetric group $\mathfrak{S}_n$ on $n$ elements, which acts on $\C[\lambda_1,\ldots,\lambda_n]$. From the above definition, it is obvious that $\Delta_d$ is invariant under this action. Furthermore, any permutation acts on $\tilde{\Delta}_d$ only by a sign. Since $\mathfrak{S}_n$ only has two one-dimensional representations, either $\tilde{\Delta}_d$ is symmetric, or it is antisymmetric.

Any symmetric polynomial can be written as a polynomial in the symmetric power sum polynomials
\[
\lambda_1+\cdots+\lambda_n,\ \lambda_1^2+\cdots+\lambda_n^2,\ldots,\ \lambda_1^n+\cdots+\lambda_n^n\,.
\]
Since for every $k\in[n]$ the matrix $A^k$ has eigenvalues precisely $\lambda_1^k,\ldots,\lambda_n^k$, we recognize the symmetric power sums as traces of powers of $A$; in particular, they are elements of $\C[A]$:
\[
    \C[\lambda_1,\ldots,\lambda_n]^{\mathfrak{S}_n}=\C[\text{tr}(A),\ldots,\text{tr}(A^n)]\subseteq \C[A]\,.
\]
We consider two cases: If $\tilde{\Delta}_d$ is symmetric, then by the above, we have $\tilde{\Delta}_d\in \C[A]$ and so $\Delta_d$ is a perfect square in $\C[A]$. This means that each irreducible factor of $\Delta_d$ appears with an even power, since $\Delta$ is irreducible (by standard theory of discriminants), and this property is preserved when passing to $\Delta_d^{\sat}$.

Now assume that $\tilde{\Delta}_d$ is antisymmetric, then $\tilde{\Delta}_d+\tau\cdot\tilde{\Delta}_d=0$ for any transposition $\tau=(i,j)$. On the other hand setting $\lambda_i=\lambda_j$ we also have  $\tilde{\Delta}_d|_{\{\lambda_i=\lambda_j\}}=(\tau\cdot\Delta_d)|_{\{\lambda_i=\lambda_j\}}$. This is only possible if $\tilde{\Delta}_d|_{\{\lambda_i=\lambda_j\}}=0$ and hence $\lambda_i-\lambda_j$ divides $\tilde{\Delta}_d$. Since this holds for any $i,j$ and different choices give coprime factors, we get that $\sqrt{\Delta}$ divides $\tilde{\Delta}_d$, where
\[
     \sqrt{\Delta}\coloneqq\prod_{i<j}(\lambda_i-\lambda_j)\,.
\]
Notice that $\sqrt{\Delta}^2=\Delta$ holds up to sign. Pick a polynomial $g\in \C[\lambda_1,\ldots,\lambda_n]$ such that $g\sqrt{\Delta}=\tilde{\Delta}_d$. Since both $\tilde{\Delta}_d$ and $\sqrt{\Delta}$ are antisymmetric, $g$ must be symmetric and hence a polynomial in $\C[A]$. On the other hand, we have $g^2=\Delta\cdot\Delta_d$ up to sign and hence $\Delta\cdot\Delta_d$ is a perfect square in $\C[A]$. Since $\Delta$ is irreducible, we conclude that every irreducible factor of $\Delta_d$ which is not equal to $\Delta$ must appear to an even power, meaning $\Delta_d^{\sat}$ is a perfect square.

We now prove the second part of the statement, which states that $\sqrt{\Delta_d^{\sat}}$ divides $\det K_d(f)$. To this end, let $h$ be an irreducible factor of $\sqrt{\Delta_d^{\sat}}$ and let $\alpha$ denote its multiplicity in $\sqrt{\Delta_d^{\sat}}$. We need to show that $h^\alpha$ divides $\det K_d(f)$. Let $B\in\C^{n\times n}$ be a generic matrix at which $h$ vanishes. To show that $h^\alpha$ divides $\det K_d(f)$, it suffices to show that any partial derivative in the variables of $\C[A]$ of order at most $k-1$ vanishes at $B$.
Consider the $C^\infty$-map of differentiable manifolds
\[
    \psi\colon\widetilde{\mathrm{GL}}_n\times \C^n\to \C^{n\times n}\,,\quad (S,\lambda)\mapsto S^{-1}\text{diag}(\lambda)S\,,
\]
where $\widetilde{\mathrm{GL}}_n$ consists of all invertible matrices where every column is normalized to length 1 in the complex Euclidean norm. Since $B$ is a generic point where $h^\circ$ vanishes and since $\Delta$ does not divide $\Delta_d^{\sat}$ by construction, we can assume $\Delta(B)\neq 0$. This in particular means that $B$ has distinct eigenvalues, hence it is diagonalizable. By definition of $\widetilde{\mathrm{GL}}_n$, the map $\psi$ is locally $n!$-to-1 around $B$. Let $\kv\ni B$ be an Euclidean open neighbourhood of $B$ such that every point in $\kv$ has precisely $n!$ distinct preimages under $\psi$. By possibly shrinking $\kv$, we can assume that $\psi^{-1}(\kv)$ is a disjoint union of $n!$ open subsets of $\widetilde{\mathrm{GL}}_n\times\C^n$. Let $\ku$ be one of these subsets. The restriction $\psi\colon\ku\to\kv$ is a smooth bijective map, hence it can be locally inverted around $B$ by the inverse function theorem. Up to shrinking $\ku$ and $\kv$, we find a map $\varphi\colon\kv\to\ku$ such that $\psi\circ\varphi=\text{id}_\kv$. We get three smooth, differentiable maps out of $\varphi$:
\begin{enumerate}
     \item $\varphi^+$ is the composition of $\varphi$ with the first projection to $\widetilde{\mathrm{GL}}_n$.
     \item $\psi^-$ is the composition of $\psi^+$ with matrix inversion.
     \item $\varphi^0$ is the composition of $\varphi$ with the projection to the second factor and the embedding of $\C^n$ into $\C^{n\times n}$ as diagonal matrices.
\end{enumerate}
Summing up, for any matrix $A\in\kv$ we have $\varphi^-(A)\varphi^0(A)\varphi^+(A)=A$, $\varphi^-(A)\varphi^+(A)=I_n$, where $I_n$ is the identity matrix of size $n$, and all three maps are infinitely differentiable locally around $B$. It follows that
\[
\rho_d(A)=\rho_d(\varphi^-(A))\rho_d(\varphi^0(A))\rho_d(\varphi^+(A))\quad\text{and}\quad\rho_d(\varphi^-(A))\rho_d(\varphi^+(A))=I_N\,,
\]
where $I_N$ is the identity matrix of size $N$.
Coming back to the task at hand, we want to show that all partial derivatives of order at most $k-1$ of $\det K_d(f)$ vanish at $B$. Since this is a local question in $B$, we may restrict the polynomial function $\det K_d(f)\colon\kv\to\C$, $A\mapsto\det K_d(f)(A)$. We can now write
\begin{align*}
\det K_d(f) &= \det
    \begin{pmatrix}
        C_f\\
        C_f\rho_d(A)\\
        C_f\rho_d(A)^2\\
        \vdots\\
        C_f\rho_d(A)^{N-1}
    \end{pmatrix} =
    \det
    \begin{pmatrix}
        C_f\rho_d(\varphi^-(A))\rho_d(\varphi^+(A))\\
        C_f\rho_d(\varphi^-(A))\rho_d(\varphi^0(A))\rho_d(\varphi^+(A))\\
        C_f\rho_d(\varphi^-(A))\rho_d(\varphi^0(A))^2\rho_d(\varphi^+(A))\\
        \vdots\\
        C_f\rho_d(\varphi^-(A))\rho_d(\varphi^0(A))^{N-1}\rho_d(\varphi^+(A))
    \end{pmatrix}
    \\
    &=\det\begin{pmatrix}
        C_f\rho_d(\varphi^-(A))\\
        C_f\rho_d(\varphi^-(A))\rho_d(\varphi^0(A))\\
        C_f\rho_d(\varphi^-(A))\rho_d(\varphi^0(A))^2\\
        \vdots\\
        C_f\rho_d(\varphi^-(A))\rho_d(\varphi^0(A))^{N-1}
    \end{pmatrix}
    \cdot\det\rho_d(\varphi^+(A))\\
    &= \det\rho_d(\varphi^+(A))\cdot\det\text{Van}(\rho_d(\varphi^0(A)))\cdot\prod_{i=1}^N(C_f\rho_d(\varphi^-(A)))_i\\
    &=\sqrt{\Delta_d}\cdot\det\rho_d(\varphi^+(A))\cdot\prod_{i=1}^N(C_f\rho_d(\varphi^-(A)))_i\,,
\end{align*}
where $\text{Van}(\rho_d(\varphi^0(A)))$ is the Vandermonde matrix in the entries of the diagonal matrix $\rho_d(\varphi^0(A))$. By definition of $\varphi^0$, these entries are exactly the monomials of degree $d$ in the eigenvalues of $A$, or equivalently, the eigenvalues of $\rho_d(A)$. It is a classical result that the determinant of the Vandermonde matrix is the product of all differences of the entries, which is precisely $\sqrt{\Delta_d}$. We have $h^\alpha|\sqrt{\Delta_d}$ and all other factors in the above formula for $\det K_d(f)$ are differentiable locally around $B$ (including the possibly rational power of $\Delta$), hence by the chain rule all partial derivatives of order at most $k-1$ vanish at $B$ as claimed.
\end{proof}

Combining all of the above results, we now prove the main theorem of this section.

\begin{theorem}\label{thm: factors determinant Kalman matrix}
    Let $f\in\C[x]_d\setminus\{0\}$. For any $\mu\in P_d^{\mle n}$, let $p_{\mu}$ be the polynomial defining the hypersurface $\kk_{\mu}(f)$, up to a scalar factor. Then, up to a scalar factor,
    \begin{equation}\label{eq: factorization of det Kalman}
    \det K_d(f) = \sqrt{\Delta_d^{\sat}}\prod_{\mu\in P_d^{\mle n}}p_{\mu}\,.
    \end{equation}
    As a consequence, the polynomial defining the Kalman variety $\kk(f)$ is
    \begin{equation}\label{eq: wannabe ratio}
        \frac{\det K_d(f)}{\sqrt{\Delta_d^{\sat}}\prod_{\mu\in P_d^{\mle n}\setminus\{(d)\}}p_{\mu}}\in\C[A]\,.
    \end{equation}
\end{theorem}
\begin{proof}
First notice that $\kk_{\mu}(f)\subseteq\V(\det K_d(f))$ for every $\mu\in P_d^{\mle n}$. Indeed, if $A\in \kk_{\mu}(f)$, then $\rho_d(A)$ has an eigenvector in the hyperplane defined by $C_f$ in $\PP^{N-1}$ and therefore $\det K_d(f)(A)=0$. Therefore, we obtain that $p_{\mu}$ divides $\det K_d(f)$ for every $\mu\in P_d^{\mle n}$. By Lemma~\ref{lem: discriminant_factors_with_multi}, we know that $\sqrt{\Delta_d^{\sat}}$ divides $\det K_d(f)$. This implies that
\[
    \deg \det K_d(f) \ge \deg\sqrt{\Delta_d^{\sat}}+\sum_{\mu\in P_d^{\mle n}}\deg p_\mu\,.
\]
To conclude, it suffices to show that the previous inequality is in fact an equality. By a straightforward computation, we obtain
\begin{equation}\label{eq: deg det Kalman matrix}
    \deg\det K_d(f) = d\binom{N}{2}\,.
\end{equation}
Recall that $\Delta$ and $\Delta_d$ denote respectively the discriminants of the characteristic polynomials of $A$ and $\rho_d(A)$.
As in Lemma~\ref{lem: discriminant_factors_with_multi}, we write $\Delta_d=\Delta_d^{\sat}\cdot\Delta^k$.
In the following, we define
\begin{align*}
\kd_2 &\coloneqq \{(\alpha,\beta)\in\N_0^n\mid\text{$\alpha\prec_{\mathrm{lex}}\beta$, $|\alpha|=|\beta|=d$, and $d_H(\alpha,\beta)=2$}\}\\
\kd_{\neq 2} &\coloneqq \{(\alpha,\beta)\in\N_0^n\mid\text{$\alpha\prec_{\mathrm{lex}}\beta$, $|\alpha|=|\beta|=d$, and $d_H(\alpha,\beta)\neq 2$}\}\,,
\end{align*}
where $d_H(\alpha,\beta)\coloneqq\lvert\{i\mid\alpha_{i}\neq\beta_i\}$ is the Hamming distance between two vectors $\alpha,\beta\in\N_0^n$. Furthermore, for all $i\ge 0$ we consider the set $\km_i$ of monomials of degree $i$ in $\C[\lambda_1,\ldots,\lambda_n]$. To compute the multiplicity $k$ of $\Delta$ in $\Delta_d$, we observe that
\begin{equation}\label{eq: expand Delta d}
\begin{split}
    \Delta_d &=(-1)^{\binom{N}{2}}\prod_{(\alpha,\beta)\in\kd_{\neq 2}}(\lambda^\alpha-\lambda^\beta)^2\prod_{(\alpha,\beta)\in\kd_2}(\lambda^\alpha-\lambda^\beta)^2\\
    &=(-1)^{\binom{N}{2}}\prod_{(\alpha,\beta)\in\kd_{\neq 2}}(\lambda^\alpha-\lambda^\beta)^2\prod_{i<j}\prod_{t=1}^d\prod_{m\in \mathcal{M}_{d-t}}(\lambda_i^t-\lambda_j^t)^2m^2\\
    &=(-1)^{\binom{N}{2}}\prod_{(\alpha,\beta)\in\kd_{\neq 2}}(\lambda^\alpha-\lambda^\beta)^{2}\prod_{i<j}\prod_{t=1}^d(\lambda_i^t-\lambda_j^t)^{2\binom{n+(d-t)-1}{d-t}}\prod_{m\in \mathcal{M}_{d-t}}m^2\\
    &=(-1)^{\binom{N}{2}+\binom{n}{2}}\Delta^{\sum_{t=1}^d\binom{n+(d-t)-1}{d-t}}\prod_{i<j}\prod_{t=1}^d\left(\frac{(\lambda_i^t-\lambda_j^t)}{\lambda_i-\lambda_j}\right)^{2\binom{n+(d-t)-1}{d-t}}\cdot\\
    &\quad\cdot\prod_{(\alpha,\beta)\in\kd_{\neq 2}}(\lambda^\alpha-\lambda^\beta)^{2}\prod_{i<j}\prod_{t=1}^d\prod_{m\in \mathcal{M}_{d-t}}m^2\,.
\end{split}
\end{equation}
In particular, the multiplicity of $\Delta$ in the previous expression is
\[
    k=\sum_{t=1}^d\binom{n+(d-t)-1}{d-t}=\sum_{s=0}^{d-1}\binom{n+s-1}{s}=\binom{n+d-1}{d-1}\,,
\]
hence
\begin{equation}\label{eq: deg Delta d sat}
    \deg\sqrt{\Delta_d^{\sat}} = \frac{\deg\Delta_d -k\deg\Delta}{2}=\frac{dN(N-1)-kn(n-1)}{2}=d\binom{N}{2}-\frac{dN(n-1)}{2}\,.
\end{equation}
Moreover, by Theorem~\ref{thm: degrees generalized Kalman varieties} we have
\begin{equation}\label{eq: sum of the degree of kalman}
    \sum_{\mu\in P_d^{\mle n}}\deg p_{\mu} = \frac{(n-1)d}{2}\sum_{\mu\in P_d^{\mle n}}\binom{n}{n-(m_1+\ldots+m_d),m_1,\ldots,m_d}\,.
\end{equation}
Notice that the multinomial
\[
    \binom{n}{n-(m_1+\ldots+m_d),m_1,\ldots,m_d}
\]
is the number of monomials of degree $d$ in $n$ variables with exactly $m_i$ exponents equal to $\mu_i$ for each $1\le i\le s$ where $\mu=(\mu_1,\ldots,\mu_s)$. Since the list of exponents of a monomial of degree $d$ in $n$, up to relabeling the variables, corresponds to a partition $\mu$, we obtain that 
\begin{equation}\label{eq: number of monomials}
    \sum_{\mu\in P_d^{\mle n}}\binom{n}{n-(m_1+\ldots+m_d),m_1,\ldots,m_d}=N\,.
\end{equation}
The statement follows combining Equations~\eqref{eq: deg det Kalman matrix}, \eqref{eq: deg Delta d sat}, \eqref{eq: sum of the degree of kalman}, and~\eqref{eq: number of monomials}.
\end{proof}

\begin{example}
Let us derive again, applying \ref{thm: factors determinant Kalman matrix}, the factorization \eqref{eq: factorization det Kalman conic} of the determinant of the Kalman matrix $K_2(f)$, for a nonsingular plane conic $\V(f)\subseteq\PP^2$. Recall that the total degree of $\det K_2(f)$ is $30$, and, up to a scalar factor,
\[
\det K_2(f) = \sqrt{\Delta_2^{\sat}}\cdot p_{(2)}\cdot p_{(1,1)}\,,
\]
where $p_{(2)}$ defines the Kalman variety $\kk(f)$, while $p_{(1,1)}$ defines the $(1,1)$-Kalman variety $\kk_{(1,1)}(f)$. Applying Theorem~\ref{thm: degrees generalized Kalman varieties}, one confirms that both polynomials have degree $6$. Furthermore, the discriminant $\Delta_2$ of the characteristic polynomial of $\rho_2(A)$ has degree $60$, and the multiplicity of the discriminant $\Delta$ of the characteristic polynomial of $A$ is $4$. In particular $\Delta_d^{\sat}$ is a perfect square of degree $60-4\cdot 6=36$, hence $\deg\sqrt{\Delta_d^{\sat}}=18$. The script related to this specific example is \verb|nonlinear_Kalman_matrix_conic.m2| and is available at \cite{salizzoni2025supplementary}.\hfill$\diamondsuit$
\end{example}

\section{Singular loci of nonlinear Kalman varieties}\label{sec: singular loci}

The singular strata of linear Kalman varieties are fully described in \cite[Thms. 4.5, 4.6]{ottaviani2013matrices}. As one might expect from the previous section, most of their results do not generalize to the case of nonlinear Kalman varieties. In this section, first, we describe in Theorem~\ref{thm: decomp sing locus Kalman} an irreducible decomposition of the reduced singular locus $\mathrm{Sing}(\kk(X))$ of $\kk(X)$ and we compute the codimensions of its components. We then restrict to the case of a nonsingular hypersurface $X$, and in Theorem~\ref{thm: degree singular locus Kalman smooth hypersurface} we compute the degree of $\mathrm{Sing}(\kk(X))$, via degeneration of $X$ into a union of hyperplanes.

\begin{theorem}\label{thm: decomp sing locus Kalman}
Let $X$ be a variety whose irreducible components are $X_1,\ldots,X_k$ and let $Y_1,\ldots,Y_t$ be the irreducible components of the singular locus of $X$. Then, we have the following decomposition into irreducible components
\begin{equation}\label{eq: decomp sing locus Kalman}
    \mathrm{Sing}(\kk(X))=\bigcup_{i,j\in[k]}S_{i,j}\cup\bigcup_{\ell\in[t]}\kk(Y_\ell)\,,
\end{equation}
where $S_{i,j}\coloneqq\overline{S_{i,j}^\circ}$ for all $i,j\in[k]$ and
\begin{align*}
    S_{i,j}^\circ &\coloneqq \{[A]\in\PP^{n^2-1}\mid\text{$\exists\,x_1\in X_i\setminus X_j$ and $\exists\,x_2\in X_j\setminus X_i$ eigenpoints of $A$}\} & \forall\,i<j\in[k]\\
    S_{i,i}^\circ &\coloneqq \{[A]\in\PP^{n^2-1}\mid\text{$\exists\,x_1,x_2\in X_i$ distinct eigenpoints of $A$}\} & \forall\,i\in[k]\,.
\end{align*}
Moreover,
\[
\codim S_{i,j} = \codim X_i+\codim X_j\quad\forall\,i,j\in[k]\,,\ i\le j\,.
\]
\end{theorem}
\begin{proof}
Consider the incidence variety $\Sigma(X)$ in \eqref{eq: incidence variety Sigma}, which we rewrite below:
\[
    \Sigma(X) \coloneqq \{([A],x)\in\PP^{n^2-1}\times\PP^{n-1}\mid\text{$x\in X$ is an eigenpoint of $A$}\}\,.
\]
The singular locus of $\Sigma(X)$ is
\[
    \mathrm{Sing}(\Sigma(X)) = \left\{(x,[A])\in \Sigma(X)\mid x\in\mathrm{Sing}(X)\right\}\,.
\]
Recall that the morphism $\pi_1\colon\Sigma(X)\to\PP^{n^2-1}$ induced by the projection onto the first factor is birational over its image $\kk(X)$. In particular, for every $y\in\kk(X)\setminus\bigcup_{i,j\in[k]}S_{i,j}$, there exists an open neighborhood $\ku$ of $y$ in $\kk(X)$ for which $\pi_1$ restricted to $\pi_1^{-1}(\ku)$ is an isomorphism. Indeed, suppose that such a neighborhood does not exist. Then, we can construct a sequence of points $\{y_i\}_{i\in\N}$ that converges to $y$ and such that $\lvert\pi_1^{-1}(y_i)\rvert\ge 2$. This implies that $y\in\overline{\bigcup_{i,j\in[k]}S_{i,j}^\circ}=\bigcup_{i,j\in[k]}S_{i,j}$, a contradiction.
We conclude that
\[
    \mathrm{Sing}\left(\kk(X)\setminus \bigcup_{i,j\in[k]}S_{i,j}\right)=\pi_1(\mathrm{Sing}(\Sigma(X)))\,.
\]
Since $\pi_1(\mathrm{Sing}(\Sigma(X)))$ is the set of classes $y=[A]$ of matrices $A\in\C^{n\times n}$ with an eigenpoint $x\in\mathrm{Sing}(X)$, we obtain
\[
    \mathrm{Sing}(\kk(X))\subseteq \bigcup_{i,j\in[k]}S_{i,j} \cup \bigcup_{\ell\in[t]}\kk(Y_\ell)\,.
\]
Moreover, the fiber of $\pi_1$ over a generic point $y\in\bigcup_{i,j\in[k]}S_{i,j}$ has cardinality exactly $2$, and therefore $\pi_1^{-1}(y)$ is not connected. However, Zariski's Main Theorem states that the inverse image of a normal point under a birational projective morphism is connected \cite[Cor. III.11.4]{hartshorne1977algebraic}. We conclude that the generic point $y\in\bigcup_{i,j\in[k]}S_{i,j}$ is not normal and therefore is singular. This proves the other inclusion.
Clearly, $\kk(Y_1),\ldots,\kk(Y_t)$ are all irreducible since they are Kalman varieties of irreducible varieties. It remains to prove that $S_{i,j}$ is irreducible for all $i,j\in[k]$. We start with the case $i\neq j$. Consider the subset of $X_i\times X_j\times\PP^{n^2-1}$
\[
    W_{ij}^\circ = \{([A],x_1,x_2)\mid\text{$x_1$ and $x_2$ are eigenpoints of $A$, $x_1\in X_i\setminus X_j$, and $x_2\in X_j\setminus X_i$}\}\,,
\]
and let $W_{ij}\coloneqq\overline{W_{ij}^\circ}$. Consider the projection $\alpha_{ij}\colon W_{ij}\to X_i\times X_j$. Observe that $W_{ij}^\circ$ is the total space of a vector bundle over $(X_i\setminus X_j)\times(X_j\setminus X_i)$ of rank $2n-2$. Since by definition $X_i$ and $X_j$ are irreducible, so is their product, hence $W_{ij}$ is irreducible by the previous observation and
\[
    \dim W_{ij} = n^2-1-(2n-2)+\dim(X_i\times X_j)=n^2-1-(\codim X_i+\codim X_j)\,.
\]
Considering the projection $\alpha_1\colon W_{ij}\to\PP^{n^2-1}$, then $\image\alpha_1=S_{i,j}$, in particular $S_{i,j}$ is also irreducible.
Since the generic fiber of $\alpha_1$ is finite, we conclude that $\codim S_{i,j}=\codim X_i+\codim X_j$. To compute $\codim S_{i,i}=2\codim X_i$ for all $i\in[k]$, one can proceed similarly by considering the subset
\[
    V_i^\circ = \{([A],x_1,x_2)\mid\text{$x_1$ and $x_2$ are eigenpoints of $A$ and $(x_1,x_2)\in(X_i\times X_i)\setminus\Delta_{X_i}$}\}\,,
\]
and its Zariski closure $V_i\coloneqq\overline{V_i^\circ}$, where $\Delta_{X_i}$ is the diagonal in $X_i\times X_i$.

It remains to show that the irreducible components in \eqref{eq: decomp sing locus Kalman} are pairwise distinct. Since $Y_i\nsubseteq Y_j$ for all $i<j$, we immediately have that $\kk(Y_i)\nsubseteq\kk(Y_j)$ for all $i<j$. Let $i,j,\ell,r\in[k]$ be such that $i\notin\{\ell,r\}$ and $r\notin\{i,j\}$, and fix $x\in X_i\setminus(X_\ell\cup X_r)$. Then we can construct a matrix $A$ with one eigenpoint equal to $x$, one eigenpoint in $X_j\setminus (X_i\cup X_r)$, and without eigenpoints in $X_r$. This implies that $A$ have at most one eigenpoint in $X_\ell\cup X_r$ and so $[A]\in S_{i,j}\setminus S_{\ell,r}$. To prove that $\kk(Y_i)\nsubseteq S_{i,j}$, notice that in $\kk(Y_i)$ there is at least one matrix with only one eigenpoint on $Y_i$, while every matrix in $S_{i,j}$ has always at least two of them. Finally, for all $i,j\in [k]$ and $\ell\in[t]$, since $X_i\setminus (X_i\cap X_j)\nsubseteq Y_\ell$ and $X_j\setminus (X_i\cap X_j)\nsubseteq Y_\ell$, we conclude that $S_{i,j}\nsubseteq\kk(Y_\ell)$.
\end{proof}

\begin{proposition}\label{prop: decompirrcomp}
Let $X_1$ and $X_2$ be two irreducible hypersurfaces in $\PP^{n-1}$ that intersect transversally such that their intersection is irreducible. Then, we have the following decomposition into irreducible components
\[
    \kk(X_1)\cap\kk(X_2)=\kk(X_1\cap X_2)\cup S_{1,2}\,,
\]
where $S_{1,2}$ is defined as in Theorem~\ref{thm: decomp sing locus Kalman}. Moreover $\codim\kk(X_1\cap X_2) = \codim S_{1,2}=2$ and
\[
    \deg S_{1,2} = \left(\binom{n}{2}^2-\binom{n}{3}\right)\deg(X_1)\deg(X_2)\,.
\]
\end{proposition}
\begin{proof}
Since $X_1$ and $X_2$ are irreducible hypersurfaces in $\PP^{n-1}$ we have that $\kk(X_1)$ and $\kk(X_2)$ are irreducible hypersurfaces in $\PP^{n^2-1}$. On the one hand, $\dim\kk(X_1)+\dim\kk(X_2)\ge n^2-1$, and so all the components of $\kk(X_1)\cap\kk(X_2)$ have dimension at least $n^2-3$ and so codimension is at least $2$. On the other hand, $X_1$ and $X_2$ intersect transversally and so $\codim(\kk(X_1)\cap\kk(X_2))=2$, therefore every component of $\kk(X_1)\cap\kk(X_2)$ has codimension exactly $2$. 

Since $X_1\cap X_2$ is irreducible, by Proposition~\ref{prop: dim deg kalman} we have that $\kk(X_1\cap X_2)$ is also irreducible.
To compute the degree of $\kk(X_1\cap X_2)$, it suffices to notice that $\codim(\kk(X_1\cap X_2)\cap S_{1,2})\ge 3$ and apply Bezout's theorem.
\end{proof}

\begin{proposition}\label{prop: sing locus union hyperplanes}
Let $X=H_1\cup\cdots\cup H_d\subseteq\PP^{n-1}$ for $d$ generic hyperplanes $H_1,\ldots,H_d$. Then
\[
    \codim\mathrm{Sing}(\kk(X))=2\quad\text{and}\quad\deg\mathrm{Sing}(\kk(X))=\binom{d}{2}\binom{n}{2}^2+d\frac{3n-5}{4}\binom{n}{3}\,.
\]
\end{proposition}
\begin{proof}
By Theorem~\ref{thm: decomp sing locus Kalman}, we have 
\[
\mathrm{Sing}(\kk(X))=\bigcup_{i,j\in[d]}S_{i,j}\cup\bigcup_{i,j\in[d]}\kk(H_i\cap H_j)\,.
\]
Applying Proposition~\ref{prop: decompirrcomp} we obtain $\deg S_{i,j}=\binom{n}{2}^2-\binom{n}{3}$ for all $i<j$. Secondly we have $\deg\kk(H_i\cap H_j)=\binom{n}{3}$ for all $i<j$ by Proposition~\ref{prop: dim deg kalman}. Finally, applying \cite[Thm. 4.6]{ottaviani2013matrices} (in the reference, choose $s=2$ and $d=n-1$), one verifies that $\deg S_{i,i}=\frac{3n-5}{4}\binom{n}{3}$ for all $i\in[d]$. Summing up, we obtain
\[
\begin{split}
    \deg\mathrm{Sing}(\kk(X)) &= \sum_{i<j\in[d]}\left(\deg S_{i,j}+\deg\kk(H_i\cap H_j)\right)+\sum_{i=1}^d\deg S_{i,i}\\
    &= \binom{d}{2}\binom{n}{2}^2+d\frac{3n-5}{4}\binom{n}{3}\,,
\end{split}
\]
which is the desired formula.
\end{proof}

\begin{example}\label{ex: two lines singular locus}
Let $X=L_1\cup L_2\subseteq\PP^2$ for some distinct lines $L_1$ and $L_2$ meeting at the point $P$. Then we get the following decomposition of $\mathrm{Sing}(\kk(X))$ into four irreducible components
\begin{equation}\label{eq: decomposition singular locus kalman union two lines}
    \mathrm{Sing}(\kk(X)) = S_{1,1}\cup S_{2,2}\cup S_{1,2}\cup\kk(P)\,,
\end{equation}
all of codimension $2$ in $\PP(\C^{3\times 3})\cong\PP^8$ and with degrees $\deg S_{1,1}=\deg S_{2,2}=\deg\kk(P)=1$ and $\deg S_{1,2}=8$, giving $\deg\mathrm{Sing}(\kk(X))=11$.\hfill$\diamondsuit$
\end{example}

Our next goal is to compute $\deg\mathrm{Sing}(\kk(X))$ for a nonsingular hypersurface $X\subseteq\PP^{n-1}$, using deformation theory. First, we provide a motivating example.

\begin{example}\label{ex: nonsingular conic degeneration two lines}
Let $X=\V(f)\subseteq\PP^2$ be a nonsingular conic. Without loss of generality, we assume that
\[
    f = a\,x_1^2+x_1x_2+b\,x_2^2+c\,x_1x_3+d\,x_2x_3+e\,x_3^2\,,\quad\,(a,b,c,d,e)\in\C^5\,,
\]
in particular, the coefficient of $x_1x_2$ does not vanish and hence can be normalized to one.
Applying Example \ref{ex: two lines singular locus}, we understand the degree of the singular locus of the Kalman variety of a singular conic, that is, the union of two lines.
Our goal is to deform $X$ to a union of two lines and track the Kalman singularities along the way. To do this, we fix two lines, say $L_1=\V(x_1),L_2=\V(x_2)$. Then $X$ deforms into $L_1\cup L_2$ by letting $(a,b,c,d,e)$ go to the zero vector. More formally, we may consider the following incidence variety
\[
    V \coloneqq \{([s:t],[x_1:x_2:x_3])\in \PP^1\times \PP^2\mid sa\,x_1^2+t\,x_1x_2+sb\,x_2^2+sc\,x_1x_3+sd\,x_2x_3+se\,x_3^2=0\}\,.
\]
The projection $\pi\colon V\to\PP^1$ makes $V$ a flat family over $\PP^1$, with $\pi^{-1}([1:1])=X$ and $\pi^{-1}([0:1])=L_1\cup L_2$.

By Example \ref{ex: two lines singular locus} we have $\deg\mathrm{Sing}(\kk(L_1\cup L_2))=11$. We verified that $\deg\mathrm{Sing}(\kk(X))=10$ for $(a,b,c,d,e)\in\C^5$ generic in the script \verb|singular_locus_Kalman_conic.m2| available at \cite{salizzoni2025supplementary}. This fact is proved in the following Theorem~\ref{thm: degree singular locus Kalman smooth hypersurface}. Intuitively, we explain the discrepancy between the degrees of $\mathrm{Sing}(\kk(L_1\cup L_2))$ and $\mathrm{Sing}(\kk(X))$ as follows. Consider the irreducible decomposition of $\mathrm{Sing}(\kk(L_1\cup L_2))$ given in \eqref{eq: decomposition singular locus kalman union two lines}. The first three components $S_{1,1}$, $S_{2,2}$, $S_{1,2}$ come from matrices that have two distinct eigenpoints in $L_1\cup L_2$. We expect this part to behave nicely under flat deformation: in particular, the locus of matrices having two eigenpoints in $X$ should deform into $S_{1,1}\cup S_{2,2}\cup S_{1,2}$, hence it has the same degree. On the other hand, the last component $\kk(P)$ in \eqref{eq: decomposition singular locus kalman union two lines} comes from the singular locus of $L_1\cup L_2$ and hence does not appear for the nonsingular conic $X$. Since $\deg\kk(P)=1$ and all the components in \eqref{eq: decomposition singular locus kalman union two lines} have the same codimension, heuristically we write the formula
\[
    \deg\mathrm{Sing}(\kk(X))+\deg\kk(P)=\deg\mathrm{Sing}(\kk(L_1\cup L_2))\,,
\]
hence
\[
    \deg\mathrm{Sing}(\kk(X))=\deg\mathrm{Sing}(\kk(L_1\cup L_2))-\deg\kk(P)=11-1=10\,.\tag*{\text{$\diamondsuit$}}
\]
\end{example}

To formalize and generalize this example, we prove the following result.

\begin{proposition}\label{prop: smoothable}
Let $Z\subseteq\PP^1\times\PP^{n-1}$ be an irreducible hypersurface and let $\pi\colon Z\to\PP^1$ be the projection onto the first coordinate. Assume that $\pi$ is dominant and that $X=\pi^{-1}([1:0])$ is a hypersurface in $\PP^{n-1}$ that is nonsingular in codimension one. Let $Y=\pi^{-1}([0:1])$ and denote by $\mathrm{Sing}_1(Y)$ the union of those components of $\mathrm{Sing}(Y)$ that have codimension 1 in $Y$. Then
\[
    \deg\mathrm{Sing}(\kk(Y)) = \deg\mathrm{Sing}(\kk(X))+\deg\kk(\mathrm{Sing}_1(Y))\,.
\]
\end{proposition}
\begin{proof}
By \cite[Prop. III.9.7]{hartshorne1977algebraic}, the projection $\pi\colon Z\to\PP^1$ is flat. Since flatness is preserved under base change and composition (see \cite[Prop. III.9.2 (b) and (c)]{hartshorne1977algebraic}), the fiber product
\[
    Z' \coloneqq Z\times_{\PP^1}Z = \{(z_1,z_2)\in Z\times Z \mid \pi(v_1)=\pi(v_2)\} = \{((s_1,x_1),(s_2,x_2))\in Z\times Z\mid s_1=s_2\}
\]
is also flat. Hence every irreducible component of $Z'$ maps dominantly to $\PP^1$ (again by \cite[Prop. III.9.7]{hartshorne1977algebraic}).
Let $V$ be any irreducible component of $Z'$ and define $W\coloneqq W^\circ$, where
\begin{align*}
    W^\circ &\coloneqq \{([A],(s_1,x_1),(s_2,x_2))\in \PP^{n^2-1}\times V\mid\text{$x_1$ and $x_2$ are distinct eigenpoints of $A$}\}\,.
\end{align*}
Note that $W$ is irreducible as it is the Zariski closure of a vector bundle over an open subset of $V$. Furthermore, the projection of $W$ onto $\PP^1$ is dominant.
Denote by $\kk_V\subseteq\PP^{n^2-1}\times\PP^1$ the image of $W$ under the projection onto the first two coordinates, then also $\kk_V$ is irreducible and maps dominantly to $\PP^1$, hence it is a flat family over $\PP^1$ by \cite[Prop. III.9.7]{hartshorne1977algebraic}. Let $\kk\coloneqq\bigcup_V\kk_V$ be the union of all those flat families over all irreducible components $V$ of $Z'$. Then, by construction, every irreducible component of $\kk$ maps dominantly to $\PP^1$, hence by \cite[Prop. III.9.7]{hartshorne1977algebraic} also $\kk$ is a flat family over $\PP^1$. This implies that the fibers of $\kk$ under the projection to $\PP^1$ all have the same Hilbert polynomial \cite[Thm. III.9.9]{hartshorne1977algebraic} and, in particular, the same degree. Denote by $\kk_\infty$ the fiber over the point $[0:1]\in\PP^1$. We also note that, since $X$ is nonsingular, by Theorem~\ref{thm: decomp sing locus Kalman} the fiber of $\kk$ over $[1:0]$ is exactly $\mathrm{Sing}(\kk(X))$.
Now let us consider $\mathrm{Sing}(\kk(Y))$. By Theorem~\ref{thm: decomp sing locus Kalman} we have \begin{align*}
    \mathrm{Sing}(\kk(Y)) &= \kk(\mathrm{Sing}(Y))\cup\overline{\{[A]\in\PP^{n^2-1}\mid\text{$\exists\,x_1,x_2\in Y$ distinct eigenpoints of $A$}\}}\\
    &= \kk(\mathrm{Sing}(Y))\cup \kk_\infty\,,
\end{align*}
where the two sets in the above union share no irreducible components. On the one hand, $\kk(\mathrm{Sing}(Y))$ can have irreducible components of different dimensions, depending on the decomposition of $\mathrm{Sing}(Y)$. On the other hand, the second set $\kk_\infty$ in the union is always equidimensional of codimension one in $\kk(Y)$ (the codimension is one since $Z$ is a flat family and hence $Y$ is of codimension one in $\PP^{n-1}$). For the computation of the degree of $\mathrm{Sing}(\kk(Y))$, only the components of codimension one in $\kk(Y)$ contribute; those are exactly the Kalman varieties of components of $\mathrm{Sing}(Y)$ which have codimension one in $Y$. Hence
\[
    \deg\mathrm{Sing}(\kk(Y)) = \deg\kk(\mathrm{Sing}_1(Y))+\deg \kk_\infty = \deg\kk(\mathrm{Sing}_1(Y))+\deg\mathrm{Sing}(\kk(X))\,,
\]
where the second equality uses that $\kk$ is a flat family.
\end{proof}

\begin{theorem}\label{thm: degree singular locus Kalman smooth hypersurface}
Let $X\subseteq \PP^{n-1}$ be a nonsingular hypersurface of degree $d$. Then
\[
    \codim\mathrm{Sing}(\kk(X))=2\quad\text{and}\quad\deg\mathrm{Sing}(\kk(X)) = \binom{d}{2}\binom{n}{2}^2+d\frac{3n-5}{4}\binom{n}{3}-\binom{d}{2}\binom{n}{3}\,.
\]
\end{theorem}
\begin{proof}
It is possible to construct a flat family $Z$ over $\PP^1$ such that, considering the morphism $\pi\colon Z\to\PP^1$, then $\pi^{-1}([1:0])=X$ and $\pi^{-1}([0:1])$ is a union of $d$ hyperplanes in general position, and all the other varieties in the family are nonsingular projective hypersurfaces isomorphic to $X$. The statement follows applying Proposition~\ref{prop: sing locus union hyperplanes} and Proposition~\ref{prop: smoothable}.
\end{proof}

\begin{example}
If $X\subseteq \PP^2$ is a nonsingular plane curve of degree $d$, then
\[
\deg\mathrm{Sing}(\kk(X))=d(4d-3)\,.
\]
Fixing $d=2$ instead, then $X\subseteq\PP^{n-1}$ is a nonsingular quadric hypersurface and
\[
\deg\mathrm{Sing}(\kk(X)) = \binom{n}{2}^2+\frac{3n-7}{2}\binom{n}{3}\,.
\]
In particular, this confirms the computation given in Example \ref{ex: nonsingular conic degeneration two lines} for $n=3$. Another interesting example is the Grassmannian $X=\G(1,3)$ of lines in $\PP^3$, which can be realized as a nonsingular quadric hypersurface in $\PP(\bigwedge^2\C^4)\cong\PP^5$. In particular, the Grassmannian $\G(1,3)$ is a particular truncation variety $V_{\{1\}}$ in \cite{faulstich2025algebraic}, and its Kalman variety $\kk(\G(1,3))$ corresponds to the locus of $6\times 6$ matrices $A$ such that at least one of the solutions of the corresponding CC equations is an eigenpoint of $A$. In this case, $\kk(\G(1,3))$ is a hypersurface of degree $12$, while its singular locus has codimension two and degree $335$. Furthermore, a generic matrix in $\mathrm{Sing}(\kk(\G(1,3)))$ contains exactly two distinct eigenpoints on $\G(1,3)$. This case is rather special because it is possible to find elements of $\kk(\G(1,3))$ all of whose eigenpoints belong to $\G(1,3)$. Similarly as in Definition~\ref{def: symmetric power matrix}, one might consider the skew-symmetric representation $\xi_2\colon\mathrm{GL}(\C^4)\to\mathrm{GL}(\bigwedge^2\C^4)$ of $\mathrm{GL}(\C^4)$, and for every matrix $A\in\C^{4\times 4}$ one may define $\xi_2(A)\in\C^{6\times 6}$ as the {\em multiplicative compound matrix} of $A$. Then, given two distinct eigenvectors $v_1$ and $v_2$ of $A$, one verifies that $v_1\wedge v_2$ is an eigenvector of $\xi_2(A)$, and $[v_1\wedge v_2]\in\G(1,3)$. In this way, one verifies that all eigenvectors of $\xi_2(A)$ belong to $\G(1,3)$. As a consequence, the variety of multiplicative compound matrices is (strictly) contained in $\mathrm{Sing}(\kk(\G(1,3)))$. Considering also the induced map of Lie algebras $\xi_2'\colon\mathfrak{gl}(\C^4)\to\mathfrak{gl}(\bigwedge^2\C^4)$, for every matrix $A\in\C^{4\times 4}$ one may define $\xi_2'(A)\in\C^{6\times 6}$ as the {\em additive compound matrix} of $A$. The variety of $6\times 6$ additive compound matrices is a linear subspace, and is also contained in $\mathrm{Sing}(\kk(\G(1,3)))$. We also mention that additive compound matrices correspond to the one-body truncations of the Hamiltonian operator, written as an endomorphism of the Fock-space, see \cite[\S3]{sverrisdottir2025algebraic}. This observation encourages a subtler study of deeper singular strata of nonlinear Kalman varieties.\hfill$\diamondsuit$
\end{example}

In~\cite[Thm. 4.5]{ottaviani2013matrices}, the authors proved that, if $X\subseteq\PP^{n-1}$ is a linear subspace $\PP(\ker(C))$, then $\mathrm{Sing}(\kk(X))$ is cut out by the $(n-1)\times(n-1)$ minors of the associated Kalman matrix $K(C)$. We have seen that, if $X=\V(f)\subseteq\PP^{n-1}$ is a hypersurface of degree $d>1$, the Kalman variety $\kk(X)$ is strictly contained in the variety cut out by the determinant of the Kalman matrix $K_d(f)$. It is therefore not surprising that this containment is also strict for the singular locus, as shown in the following proposition. 

\begin{proposition}
Let $X\subseteq\PP^{n-1}$ be a hypersurface of degree $d$ cut out by $f\in\C[x]_d$ with associated Kalman matrix $K_d(f)$. Then $\mathrm{Sing}(\kk(X))$ is contained in the variety cut out by the $(N-1)\times(N-1)$ minors of $K_d(f)$.
\end{proposition}
\begin{proof}
By Theorem~\ref{thm: decomp sing locus Kalman} we know that $\mathrm{Sing}(\kk(X)) = \kk(\mathrm{Sing}(X))\cup S$, where
\[
S \coloneqq \overline{\{[A]\in\PP^{n^2-1}\mid\text{$\exists\,x_1,x_2\in X$ distinct eigenpoints of $A$}\}}\,.
\]
Assume first that $[A]$ is a generic point of $S$. By assumption, there exist two distinct eigenpoints $x_1\neq x_2$ of $A$ on $X$. This means that $\nu_d(x_1)$ and $\nu_d(x_2)$ are distinct eigenpoints of $\rho_d(A)$, and they lie on the hyperplane $H_f=\PP(\ker(C_f))\subseteq\PP^{N-1}$. Therefore $\rho_d(A)$ contains a two-dimensional invariant subspace in $H_f$, hence $[\rho_d(A)]\in\mathrm{Sing}(\kk(H_f))$ by \cite[Lem. 4.1]{ottaviani2013matrices}. The variety $\mathrm{Sing}(\kk(H_f))$ is cut out by the $(N-1)\times(N-1)$ minors of $K_d(f)$ by \cite[Thm. 4.5]{ottaviani2013matrices}. In particular, $[A]$ lies on the previous locus.

Otherwise let $[A]$ be a generic point of $\kk(\mathrm{Sing}(X))$, namely there exists an eigenpoint $x\in\mathrm{Sing}(X)$ of $A$. Without loss of generality, we can assume that $x=[e_1]=[1:0:\cdots:0]$. This means that $f$ can be written as $f=\sum_{j=0}^d f_jx_1^{d-j}$ for some polynomials $f_j\in\C[x_2,\ldots,x_n]_j$ with $f_0=f_1=0$. Since $[A]$ has been chosen generic, there exists another eigenpoint of $A$, call it $y=[v]$, such that $x\neq y$. Then $\nu_d(x_1)=[(e_1^*)^d]$ and $[(e_1^*)^d\cdot v^*]$ are distinct eigenpoints in $\PP(\C[x]_d)$ of $\rho_d(A)$, and by the construction of $f$ one verifies that $\{[(e_1^*)^d],[(e_1^*)^d\cdot v^*]\}\subseteq H_f$. Similarly to the previous case, we have that $\rho_d(A)$ contains a two-dimensional invariant subspace in $H_f$, and we conclude that $[A]$ lives in the variety cut out by the $(N-1)\times(N-1)$ minors of $K_d(f)$.
\end{proof}

\begin{remark}
The variety cut out by the $\left(N-1\right)\times\left(N-1\right)$ minors of $K_d(M)$ has always codimension one. This is a consequence of the fact that the determinant of $A$ divides $\det K_d(f)$ with multiplicity larger than one. More precisely, we have that $(\det A)^s|\det K_d(M)$, where
\[
    s = \sum_{t=1}^d\binom{d-t+n-2}{d-t}\left[\frac{t}{2}\left(\binom{d-t+n-2}{d-t}-1\right)+\sum_{i=1}^{t-1}\binom{d-i+n-2}{d-i}i\right]\,.
\]
The computation of $s$ comes from \eqref{eq: expand Delta d}. More precisely, one needs to extract the largest power of $\lambda_i$ (the choice of $i$ is irrelevant by symmetry) from the expression at the right-hand side of the first identity in \eqref{eq: expand Delta d}. The previous expression simplifies to $s=3\binom{d+3}{5}$ for $n=3$. In particular, if also $d=2$, then $s=3$ as in \eqref{eq: factorization det Kalman conic}.
\end{remark}

\section*{Acknowledgements}

We thank the authors of \cite{borovik2025numerical}, Leonie Kayser, and Svala Sverrisd\'ottir for the fruitful discussions and the valuable feedback received.
We would also like to thank Bernd Sturmfels for suggesting the idea of the project.
F. S. is supported by the P500PT-222344 SNSF project. J. W. is supported by the SPP 2458 ``Combinatorial Synergies'', funded by the Deutsche Forschungsgemeinschaft (DFG, German Research Foundation), project ID: 539677510.

\bibliographystyle{alphaurl}
\bibliography{biblioKalman}
\end{document}